\documentclass[ 12pt]{amsart}
\usepackage{amsfonts}
\usepackage{mathrsfs}
\usepackage{bbm}
\usepackage{amsfonts,amsthm}
\usepackage{amssymb,amsmath,graphicx}
\usepackage{cases}
\usepackage{float}
\usepackage[colorlinks=true]{hyperref}
\hypersetup{urlcolor=blue, citecolor=red}
\usepackage{hyperref}
\usepackage{enumerate}
\usepackage{cite}
\newtheorem{theorem}{Theorem}

\newtheorem{lemma}{Lemma}
\newtheorem{proposition}{Proposition}

\newtheorem{remark}{Remark}

\def\Xint#1{\mathchoice
	{\XXint\displaystyle\textstyle{#1}}
	{\XXint\textstyle\scriptstyle{#1}}
	{\XXint\scriptstyle\scriptscriptstyle{#1}}
	{\XXint\scriptscriptstyle\scriptscriptstyle{#1}}
	\!\int}
\def\XXint#1#2#3{{\setbox0=\hbox{$#1{#2#3}{\int}$ }
		\vcenter{\hbox{$#2#3$ }}\kern-.6\wd0}}

\def\dashint{\Xint-}

\def\Div{{\rm div}}
\begin{document}
	\title[Existence and Optimal Convergence Rate]{Existence and Optimal Convergence Rates of Multi-dimensional Subsonic Potential Flows Through an Infinitely Long Nozzle with an Obstacle Inside}
\author{Lei Ma}
\address{School of Mathematical Sciences, Shanghai Jiao Tong University, 800 Dongchuan Road, Shanghai, 200240, China}
\email{yutianml@sjtu.edu.cn}

	\author{Chunjing Xie}
\address{School of mathematical Sciences, Institute of Natural Sciences,
Ministry of Education Key Laboratory of Scientific and Engineering Computing,
and SHL-MAC, Shanghai Jiao Tong University, 800 Dongchuan Road, Shanghai, China}
\email{cjxie@sjtu.edu.cn}

	\date{}
	\maketitle
	\begin{abstract}
		In this paper, the well-posedness and optimal convergence rates of subsonic irrotational flows through a three dimensional infinitely long nozzle with a smooth obstacle inside are established. More precisely, the global existence and uniqueness of the uniformly subsonic flow are obtained via variational formulation as long as the incoming mass flux is less than a critical value.
		 Furthermore, with the aid of delicate choice of weight functions, we prove the {\it optimal} convergence rates of the flow at far fields via weighted energy estimates and Nash-Moser iteration.\\[3mm]{\bf Keywords:}   Subsonic flows, Potential equation, Nozzles, Optimal convergence rates, Body
	\end{abstract}
	\section{Introduction}
	
	The study on compressible inviscid flows provides many significant and challenging problems. The flows past a body, through a nozzle, and past a wall are typical flow patterns which have physical significance and also include the physical effects \cite{Bers1958Mathematical}. The first rigorous mathematical analysis on the problem for irrotational flows past a body is due to Frankl and Keldysh\cite{Frankl1934Keldysh}.
	The important progress for subsonic flows with prescribed circulation was made by Shiffman \cite{shiffman1952existence} via the variational approach. Later on, Bers \cite{MR65334} proved the existence of two dimensional irrotational subsonic flows around a profile with a sharp trailing edge if the free stream Mach number is less than a critical value. The uniqueness and asymptotic behavior of the subsonic plane flows were established in \cite{CPA:CPA3160100102}. The study of three dimensional irrotational flows around a smooth body was initiated in \cite{Finn1957Three} by Finn and Gilbarg. For the three (or higher) dimensional case, the existence of uniform subsonic irrotational flows around a smooth body was established by Dong  \cite{Dong1977Nonlinear, Dong1993Subsonic} in a weighted function space as long as the incoming Mach number is less than a critical value. When the vorticity of the flow is not zero, the Euler system for subsonic solutions is a hyperbolic-elliptic coupled system so that the problem for flows past a body becomes much harder. The well-posedness theory for two dimensional subsonic flows past a wall or a symmetric body was established in \cite{Chen2016}.  As far as transonic flows past a profile is concerned, it was proved by Morawetz \cite{morawetz1956non, morawetz1957non, morawetz1958non, morawetz1964non} that the smooth transonic flows past an airfoil are usually unstable with respect to the perturbations of the physical boundaries. Later on, Morawetz initiated a program to prove the existence of weak solutions for irrotational flows by the method of compensated compactness \cite{morawetz1985weak, morawetz1995steady}.
The compensated compactness method was successfully used to deal with subsonic-sonic flows recently, see \cite{chen2007two, chen2016subsonic, Huang2011ON, Xie2007GLobal} and reference therein.

The irrotational flows through infinitely long nozzles were first studied in \cite{Xie2007GLobal, XIE20102657} where the authors established the existence and uniqueness of global subsonic flows through two dimensional or three dimensional axially symmetric nozzles as long as the flux is less than a critical value. The existence and uniqueness of the irrotational uniformly subsonic flows in the multidimensional nozzle were established in \cite{Du2011} by the variational method. For subsonic flows with nonzero vorticity, the existence of solutions in two dimensional nozzles was first proved in \cite{Xie2009Existence} when the mass flux is less than a critical value and the variation of Bernoulli's function is sufficiently small. Furthermore, the existence of two dimensional subsonic flows and their optimal convergence rates at far fields were established in \cite{Du2014} for a large class of subsonic flows with large vorticity. Later on, the existence of general two dimensional subsonic flows even with characteristic discontinuity was proved in \cite{chen2017steady}. Subsonic flows with non-zero vorticity through infinitely long nozzles were also studied in various settings, such as axially symmetric flows, two dimensional periodic flows, etc, see\cite{Chen2012Existence,XIE20102657,MR3186843,du2014subsonic} and reference therein. Recently, the subsonic-sonic flow in a convergent nozzle with straight solid walls was studied in \cite{wang2013degenerate} and the properties of the sonic curve were investigated in \cite{wang2016sonic}. We would like to mention that there are important progress on stability of transonic shocks in nozzles, see \cite{xin2005transonic, xin2008three, MR1969202} and reference therein., where the key issue is to study subsonic solutions around some background solutions with shocks as free boundary.

Note that, the wind tunnel experiment can be regarded as the problem for flows past an obstacle in a nozzle. Our ultimate goal is to study the well-posedness theory for multidimensional subsonic flows past a non-smooth body (an open problem posed in \cite{Dong1977Nonlinear}) and through nozzles with a non-smooth body inside. As a first step, in this paper, we study the well-posedness and the optimal convergence rates at far fields for muiti-dimensional subsonic flows through nozzles with a smooth body inside.


	Consider the isentropic compressible Euler equations as follows
	\begin{equation}\label{Eulereq}
	\begin{cases}
	\Div(\rho  {\bf u}) = 0,   \\
	\Div(\rho {\bf u}\otimes {\bf u})+\nabla p=0,
	\end{cases}
	\end{equation}
	where $\rho$ represents density, ${\bf u}=(u_1,u_2,u_3)$ is the flow velocity and $p$ is the pressure given by the equation of states $p=p(\rho)$. In this paper, we always assume $p^{\prime}(\rho)>0$ and $p^{\prime\prime}(\rho)\geq0$ for $\rho>0$. For the polytropic gas, the pressure is given by $p=A\rho^{\gamma}$, where $A$ is a positive number and $\gamma>0$ is called the adiabatic exponent.
	
	Suppose that the flow is irrotational, i.e.,
	\begin{equation}\label{irrotational_property}
	\nabla \times {\bf u}=0.
	\end{equation}
	Thus, there exists a potential function $\phi$ such that
	\begin{equation}\label{potential}
	{\bf u}(x)=\nabla\phi.
	\end{equation}
	With the aid of \eqref{Eulereq} and \eqref{irrotational_property}, the following Bernoulli's law holds for irrotational flows (\cite{1948sfswbookC}),
	\begin{equation}\label{bernoulli law}
	\frac{1}{2}|{\nabla\phi}|^2+h(\rho)\equiv C  ,
	\end{equation}
	where $C$ is a constant and $h(\rho)$ is the enthalpy defined by $h^\prime(\rho)=\frac{p^\prime(\rho)}{\rho}$. It follows from \eqref{bernoulli law} that $\rho$ can be written as $\rho(|\nabla\phi|^2)$. Using the mass conservation in \eqref{Eulereq}, the Euler equations can be reduced to the potential equation
	\begin{equation}\label{divergence_form}
	\Div(\rho(|\nabla\phi|^2)\nabla\phi)=0.
	\end{equation}
	Denote $c(\rho)=\sqrt{p^{\prime} (\rho)}$ which is called the sound speed. It is easy to check that when $|{\bf u}|>c(\rho)$ (i.e. the flow is supersonic), the equation \eqref{divergence_form} is hyperbolic; while if $|{\bf u}|<c(\rho)$ (i.e. the flow is subsonic), the equation \eqref{divergence_form} is elliptic. Moreover, there is a critical speed $q_{cr}$ such that $|{\bf u}|<c(\rho)$ if and only if $|{\bf u}|<q_{cr}$  (\cite{1948sfswbookC}). Thus one can normalize $(\rho, {\bf u})$ as follows
	\begin{equation}
	{\bf\hat{u}}=\frac{{\bf u}}{q_{cr}}\quad\text{and}\quad \hat{\rho}=\frac{\rho}{\rho(q_{cr}^2)}.
	\end{equation}
	With an abuse of the notation, we still use ${\bf u}$ and $\rho$ rather than $\bf \hat u$ and $\hat \rho$ later. Denote $q=|\bf u|$. It is easy to see that $\rho q\leq1$ for $q\geq0$ and the subsonic flow means $|{\bf u}|<1$ or $\rho>1$.
	
	We consider the domain to be a nozzle $\tilde\Omega$ which contains an obstacle $\Omega'$ inside. By using the cylindrical coordinates, $\tilde\Omega$ and $\Omega'$ can be written as
	\begin{equation}\label{Omega}
	\tilde\Omega=\bigg\{(r,\theta,x_3)\big|\, r< f_1(\theta,x_3),\,\theta\in[0,2\pi),\,x_3\in\mathbb{R}\bigg\}
	\end{equation} and
	\begin{equation}\label{Omega_prime} \Omega^\prime=\bigg\{(r,\theta,x_3)\big|\, r< f_2(\theta,x_3),\,\theta\in[0,2\pi),\,L_1\leq x_3\leq L_2\bigg\},
	\end{equation}
	respectively, where $L_1$ and $L_2$ are constants.
	Assume
	\begin{equation}
	0\leq f_2\leq C\quad \text{and}\quad\frac{1}{C}\leq f_1-f_2\leq C\quad\text{for any}\, \theta\in [0,2\pi),\,x_3\in [L_1,L_2],
	\end{equation}
	and
	\begin{equation}
	\frac{1}{C}\leq f_1\leq C\quad\text{for any}\, \theta\in [0,2\pi),\,x_3\in\mathbb{R},
	\end{equation}
	where $C$ is a positive constant. Without loss of generality, assume the origin $O\in\Omega'$. Moreover, suppose that $\partial\tilde\Omega$ and $\partial\Omega'$ are $C^{2,\alpha}$. In the rest of the paper, denote
		\begin{equation}
	{\Omega}=\tilde\Omega\setminus\Omega'\quad\text{and}\quad	\Sigma_t=\Omega\cap\{x_3=t\}.
	\end{equation}
We consider subsonic flows in $\Omega$ which satisfy the slip boundary conditions on the solid walls. The problem can be formulated as follows
	\begin{equation}\label{Potentialequation}
	\begin{cases}
	\Div(\rho(|\nabla\phi|^2)\nabla\phi)=0     &\text{in} \ \Omega,\\[2mm]
	\frac{\partial\phi}{\partial\textbf{n}}=0 &\text{on}\ \partial\Omega, \\[2mm]
	\int_{\Sigma_t}\rho(|\nabla\phi|^2)\frac{\partial\phi}{\partial\textbf{l}}ds=m_0,\\[2mm]
	|\nabla\phi|<1,
	\end{cases}
	\end{equation}
    where $\textbf{n}$ is the unit outer normal of $\Omega$ and $\textbf{l}$ is the unit normal pointed to the right of $\Sigma_t$, respectively. $m_0$ is the mass flux of the flow across the nozzle, which is conserved through each cross section.
     \begin{figure}[tb]
    	\centering	\includegraphics[width=0.8\textwidth]{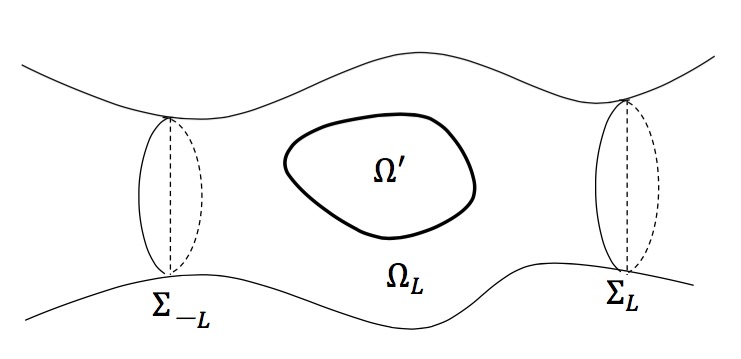}
    	\caption {Domain of the problem\label{fig:1}}
    \end{figure}

Our first main results can be stated as follows.
	\begin{theorem}\label{theorem_1}
		There exists a critical value $\hat{m}$ such that

		(i) if the mass flux $m_0<\hat{m}$, there exists a uniformly subsonic flow through ${\Omega}$, i.e. there exists a solution $\phi$ which solves the problem \eqref{Potentialequation} and satisfies
		\begin{equation}
		Q(m_0)=\sup_{\overline{\Omega}}|\nabla\phi|<1.
		\end{equation} Moreover,
		\begin{equation}
		\|\nabla\phi\|_{C^{1,\alpha}({\Omega})}\leq Cm_0,
		\end{equation}
		where $C$ is a constant independent of $m_0$.

		(ii) The value of  $Q(m_0)$ ranges over $[0,1)$ as $m_0$ varies in $[0,\hat{m})$.
	\end{theorem}

	Furthermore, if the additional structure of the nozzle is known, we have the following optimal convergence rates of the flows at far fields.
	\begin{theorem}\label{convergence_thm} Let $\overline\Sigma=\{(r,\theta)|\,r\leq \bar f,\,\theta\in[0,2\pi)\}$ with positive constants $\bar f$ and $\bar q$ satisfying $\rho(\bar q^2)\bar q=m_0/(2\pi\bar f^2)$ and $\bar q\leq c\big({\rho(\bar q^2)}\big)$.

		 (i) If the nozzle is a straight cylinder in the downstream, i.e. $\Omega\cap\{x_3\geq K\}=\overline\Sigma\times[K,+\infty)$ for some positive $K$,
		 then there exists a positive constant $\mathfrak d $ such that
		\begin{equation}\label{cylinder_decay}
		|\nabla\phi-(0,0, \bar q)|\leq Ce^{-{\mathfrak d x_3}} \quad\text{for $x\in \Omega\cap\{x_3\geq K\}$},
		\end{equation}
		 where $C$ is a constant independent of $x_3$.

		(ii) If there exists a $K>0$ such that
		\begin{equation}\label{boundary_algebratic_rate}
		\sum\limits_{k=0}^{2}\big|x_3^k\partial^k(f_1-\bar{f})\big|\leq\frac{C}{x_3^l}\quad\text{for $x_3>K$},
		\end{equation}
		with some constant $l>0$, then the velocity field satisfies
		\begin{equation}\label{algebra_decay}
		|\nabla\phi-(0,0,\bar q)|\leq C x_3^{-{l}}\quad\text{for $x\in \Omega\cap\{x_3\geq K\}$},
		\end{equation}
		where $C$ is a constant independent of  $x_3$.
	 Similarly, if the boundary of the nozzle satisfying the same asymptotic behavior as \eqref{boundary_algebratic_rate} at the upstream, the same conclusion as \eqref{algebra_decay} holds at the upstream.
	\end{theorem}
There are few remarks in order.
\begin{remark}
	The convergence rates \eqref{cylinder_decay} and \eqref{algebra_decay} do not depend on $\Omega'$. Hence the convergence rates also hold for subsonic flows in nozzles obtained in \cite{Du2011} .
\end{remark}
\begin{remark}\label{optimalremark}
	The convergence rate \eqref{algebra_decay} is optimal. Indeed, suppose that there exists a constant $C$ such that
    \begin{equation}
    |f_1-\bar f| =\frac{C}{x_3^l}\quad\text{for $x_3>K$}.
    \end{equation}
    It follows from the definition of $\bar q$ that
	\begin{equation}\label{equalflux}
	\begin{split}
	0&=\int_{\Sigma_{x_3}}\rho(|\nabla\phi|^2)\partial_3\phi dx'-\int_{\overline\Sigma}\rho{(\bar q^2})\bar qdx'\\
	&=\int_{\Sigma_{x_3}}\rho(|\nabla\phi|^2)\partial_3\phi-\rho{(\bar q^2})\bar qdx'+\pi(f^2-\bar f^2)\rho(\bar q^2)\bar q.
	\end{split}
	\end{equation}
		The straightforward computations yield that
	\begin{equation}\label{optimal}
	\begin{split}
    &\quad\int_{\Sigma_{x_3}}\rho(|\nabla\phi|^2)\partial_3\phi-\rho{(\bar q^2})\bar qdx'\\
    &=\int_{\Sigma_{x_3}}\bigg(\rho(|\nabla\phi|^2)-\rho(\bar q^2)\bigg)\bar q+\rho(|\nabla\phi|^2)(\partial_3\phi-\bar q)dx'\\
    &=\int_{{\Sigma}_{x_3}}\bar q\bigg[\rho'(\bar q^2)\big(|\nabla\phi|^2-\bar q^2\big)+O\bigg(\big(|\nabla\phi|^2-\bar q^2\big)^2\bigg)\bigg]+\rho(|\nabla\phi|^2)(\partial_3\phi-\bar q)dx'\\
    &=\int_{\Sigma_{x_3}}\big[\bar q\rho'(\bar q^2)(\partial_3\phi+\bar q)+\rho(|\nabla\phi|^2)\big](\partial_3\phi-\bar q)+\bar q\rho'(\bar q^2)\big[|\partial_1\phi|^2+|\partial_2\phi|^2\big]\\
    &\quad\quad+ O\bigg(\big(|\nabla\phi|^2-\bar q^2\big)^2\bigg).
	\end{split}
	\end{equation}
Combining \eqref{algebra_decay}, \eqref{equalflux} and \eqref{optimal}
 yields that there exists a constant $\tilde{C}$ such that
 \begin{equation}
\max\limits_{x\in\Sigma_{x_3}}|\partial_3\phi-\bar q|\geq \frac{\tilde{C}}{x_{3}^l} \quad\text{for $x_3$ sufficiently large.}
\end{equation}
This implies that the convergence rate \eqref{algebra_decay} is optimal.\end{remark}

\begin{remark}
	Applying the compensated compactness framework developed in \cite[Theorem 2.1]{Huang2011ON}, one can obtain the existence of the weak subsonic-sonic solutions through an infinitely long nozzle with a body inside.
	\end{remark}
	
   Here we give the key ideas and comments on main techniques for the proof of Theorems \ref{theorem_1} and \ref{convergence_thm}. The existence of weak solutions is obtained via the variational method inspired by \cite{Du2011} where the major new difficulty is the average estimate presented in Lemma \ref{lemma1}. The regularity of weak solutions is improved since the subsonic potential flows are governed by elliptic equations.  The key issue to prove convergence rates of subsonic flows at far fields is to study the asymptotic behavior of gradient of solutions to quasilinear elliptic equations.
   The first difficulty to study the convergence rates is that the domain is of cylindrical type so that it is not easy to use Kelvin transformation to study the asymptotic behavior as what has been done for flows past a body. The another difficulty is that the potential function is not bounded in $L^{\infty}-$norm so that it is hard for us to adapt the approach developed in \cite{Du2014} for two dimensional flows which is based on the maximum principle. Some studies on far fields behavior for solution of elliptic equations in cylindrical domains can be found in \cite{lax1957phragmen,flavin1992asymptotic,MR3405540}. Inspired by the work \cite{Ole1981ON}, we combine the convergence rates of the boundaries and the weighted energy estimate with delicate choice of weight to get an $L^2-$ decay of gradients of the velocity potential.  $L^\infty-$norm of the $\nabla\phi$ is established via Nash-Moser iteration.

The rest of this paper is organized as follows. In Section $2$, we adapt the variational method in \cite{Du2011}  to establish the existence of subsonic solutions. In Sections $3$, the optimal convergence rates of velocity at far fields are established..

	\section{Existence and uniqueness of subsonic solution with small flux}
	 In order to deal with the possible degeneracy near sonic state, we first study the problem with subsonic truncation so that the truncated equation is uniformly elliptic. The key ingredient is a priori estimate for the truncated domain.
	\subsection{Subsonic solutions of the truncated problem}
	When $|\nabla\phi|$ goes to $1$, the potential equation \eqref{divergence_form} is not uniformly elliptic. Another difficulty for the problem \eqref{Potentialequation} is that the domain ${\Omega}$ is not bounded. To overcome these difficulties, we truncate both the coefficients and the domain.
	Define $H_\epsilon(s^2)$ and $F_\epsilon(q^2)$ as follows
	\begin{equation}\label{densitytruncation}
	H_\epsilon(s^2)=
	\begin{cases}
	\rho(s^2) &\ \text{if $s^2<1-2\epsilon$},\\
	\text{smooth and decreasing} &\ \text{if $1-2\epsilon\leq s^2\leq1-\epsilon$},\\
	\rho(1-\frac{3\epsilon}{2}) &\ \text{if $s^2\geq 1-\epsilon$}
	\end{cases}
	\end{equation}
	and
	\begin{equation}
	F_\epsilon(q^2)=\frac{1}{2}\int_0^{q^2}H_\epsilon(\tau)d\tau,
	\end{equation}
	where $\epsilon$ is a small positive constant. One can easily check that there exists a positive constant $C(\epsilon)$ depending on $\epsilon$ such that
	\begin{equation}\label{coefficient_property}
	\frac{1}{C(\epsilon)}q^2\leq F_\epsilon(q^2)\leq C(\epsilon)q^2 \quad  \text{and}\quad \frac{1}{C(\epsilon)}\leq H_\epsilon(s^2)+2H_\epsilon'(s^2)s^2<H_\epsilon(s^2)<C(\epsilon).
	\end{equation}
	Denote
	\begin{equation}
	a_{ij}(\nabla\phi)=H_\epsilon(|\nabla\phi|^2)\delta_{ij}+2H_\epsilon'(|\nabla\phi|)\partial_i\phi\partial_j\phi.
	\end{equation}
	It is easy to see that there exist two positive constants $\lambda$ and $\Lambda$ such that
	\begin{equation}\label{elliptic_property}
	\lambda|\xi|^2<a_{ij}\xi_i\xi_j<\Lambda|\xi|^2,\quad\text{for any $\xi\in\mathbb{R}^3$},
	\end{equation}
	where the repeated indices mean the summation for $i$, $j$ from $1$ to $3$. This  convention is used in the whole paper.
	
For any sufficiently large positive number $L$ and any set $U$, denote
	\begin{equation*}
	\Omega_L=\Omega\cap\big\{|x_3|<L\big\}\quad\text{and} \quad\dashint_Ufdx=\frac{1}{|U|}\int_Ufdx,\quad\end{equation*} where $f\in L^1(U)$. Later on, the following notations will be used
	\begin{equation*}\overline{\text{S}}=\inf\limits_{t\in\mathbb{R}}|{\Sigma}_{t}| \quad \text{and}\quad \underline{\text{S}}=\sup\limits_{t\in\mathbb{R}}|{\Sigma}_{t}|.
	\end{equation*}
	
	In order to study the problem \eqref{Potentialequation}, we consider the $\textbf{truncated problem}$:
	\begin{equation}\label{truncated_equation}
	\begin{cases}
	\Div(H_\epsilon(|\nabla\phi|^2)\nabla\phi)=0 &\text{in }{\Omega}_L, \\
	\frac{\partial\phi}{\partial \textbf{n}}=0 &\text{on } \partial\Omega_L,\\
	H_\epsilon(|\nabla\phi|^2)\frac{\partial\phi}{\partial x_3}=\frac{m_0}{|\Sigma_{L}|}  &\text{on }\Sigma_{L},\\
	\phi=0  &\text{on} \,\Sigma_{-L}.
	\end{cases}
	\end{equation}
   Define the space
	\begin{equation}
	\mathcal H_L=\{\phi\in H^1({\Omega}_L):\phi=0\text{ on } \Sigma_{-L}\}.
	\end{equation}
	It is easy to see that $\mathcal H_L$ is a Hilbert space under $H^1$ norm.
	 $\phi$ is said to be a weak solution of the problem \eqref{truncated_equation} in $\mathcal H_L$ if
	\begin{equation}\label{weak_solution}
	\int_{{\Omega}_L}H_\epsilon(|\nabla\phi|^2)\nabla\phi\cdot\nabla\psi dx-\frac{m_0}{|\Sigma_{L}|}\int_{\Sigma_{L}}\psi dx'=0,\quad\text{for any $\psi\in \mathcal H_L.$}
	\end{equation}
	Define
	\begin{equation}\label{functional}
    \mathcal{I}_L(\psi)=\int_{{\Omega}_L}F_\epsilon(|\nabla\psi|^2)dx-\frac{m_0}{|\Sigma_{L}|}\int_{\Sigma_{L}}\psi dx'.
	\end{equation}
	The straightforward calculations show that if $\phi$ is a minimizer of $\mathcal I_L$, i.e.,
	\begin{equation}\label{minipb31}
	\mathcal I_L(\phi)=\min_{\psi\in\mathcal H_L}\mathcal I_L(\psi),
	\end{equation}
     then $\phi$ must satisfy \eqref{weak_solution}.

     First, we have the following lemma on the existence of minimizer and basic estimate for the minimizer of the problem \eqref{minipb31}.
	\begin{lemma}\label{lemma1}
		For any sufficiently large $L$, $\mathcal I_L(\psi)$ has a minimizer $\phi\in \mathcal H_L$. Moreover, this minimizer  satisfies
		\begin{equation}\label{key_estimae}
		\dashint_{{\Omega}_L}|\nabla\phi|^2dx\leq Cm_0^2,
		\end{equation}
		where $C$ is independent of $L$.
	\end{lemma}
	\begin{proof}
		Given a fixed positive constant $M>1$, let $U_M$ be a domain satisfying
		 $U_M\subset{\Omega}_M$, $U_M\cap\{x_3=\pm M\}=\Sigma_{\pm M}$ and $\Omega'\cap U_M=\emptyset.$ Define ${U}_L=U_M\cup\big(\Omega_L\setminus\Omega_M\big)$.
		Furthermore, one can choose $U_M$  such that  $\partial U_L\setminus(\Sigma_{-L}\cup\Sigma_{L})$ is $C^{2,\alpha}$. Obviously $ U_L\subset \Omega_L$  (See figure 2). Let $B_1(0)\subset \mathbb{R}^2$ be the unit disk centered at origin. Denote  $\mathscr{C}_L=B_1(0)\times\{-L\leq x_3\leq L\}$.
	It is easy to see that there exists an invertiable $C^{2,\nu}$ map $\mathcal T$: ${U}_L\rightarrow \mathscr{C}_L$, $x\rightarrow y$ satisfying\\
	(i) $\mathcal T(\partial{U}_L)=\partial \mathscr{C}_L$.\\
	(ii) For any $-L\leq k\leq L$, $\mathcal T({U}_L\cap\{x_3=k\})=B_1(0)\times\{y_3=k\}$.\\
	(iii)$ \|\mathcal T\|_{C^{2,\nu}},\ \|\mathcal T^{-1}\|_{C^{2,\nu}}\leq C$.\\
		\begin{figure}[tb]
		\centering	\includegraphics[width=0.8\textwidth]{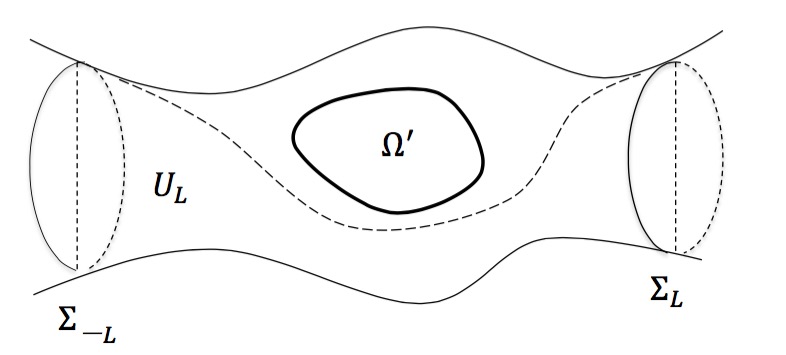}
		\caption {Domain of $U_L$\label{fig:2}}
	\end{figure}It follows from the straightforward computations that
		\begin{equation}\label{S_Lestimate}
		\begin{split}
		 \bigg|\int_{\Sigma_L}\psi dx'\bigg|&\leq C\int_{B_1(0)}|\psi(y',L)| dy'\leq C\int_{B_1(0)}\bigg(\int_{-L}^L|\partial_{y_3}\psi| dy_3\bigg)dy'\\
		 &\leq C\int_{\mathscr{C}_L}|D\psi|dy\leq C\int_{{U}_L}|\nabla\psi|dx\leq C\int_{{\Omega}_L}|\nabla\psi|dx.
		 \end{split}
		\end{equation}
		Applying H$\ddot{\text{o}}$lder inequality yields
		\begin{equation}\label{estimate_on_S_L}
		\bigg|\int_{\Sigma_L}\psi dx'\bigg|\leq C |{\Omega}_L|^{\frac{1}{2}}\|\nabla\psi\|_{L^2(\Omega_L)}.
		\end{equation}
		The constant $C$ here and subsequently in the rest of the paper may change from line to line as long as what these constants depend on is clear.
		Substituting the estimate \eqref{estimate_on_S_L} into \eqref{functional}, one can conclude that
		\begin{equation}\label{coercive}
		\begin{split}
		\mathcal I_L(\psi)&=\int_{{\Omega}_L}F_\epsilon(|\nabla\psi|^2)dx-\frac{m_0}{|\Sigma_{L}|}\int_{\Sigma_{L}}\psi dx'\\
		&\geq\lambda\int_{{\Omega}_L}|\nabla\psi|^2dx-C'\|\nabla\psi\|_{L^2({\Omega}_L)}\\
		&\geq\frac{\lambda}{2}\|\nabla\psi\|_{L^2({\Omega}_L)}^2-\frac{1}{\lambda}C',
		\end{split}
		\end{equation}
		where $C'$ depends on $m_0$, $\overline{\text{S}}$, $\underline{\text{S}}$, and $|{\Omega}_L|$.
		This implies that the functional $\mathcal I_L(\psi)$ is coercive.
	   Hence $\mathcal I_L(\psi)$ has an infimum. Let $\{\phi_k\}$ be a minimizing sequence.
		One has
		\begin{equation}
		\|\nabla\phi_k\|_{ L^2({\Omega}_L)}^2\leq\frac{2}{\lambda}\mathcal I_L(\phi_k)+\frac{2}{\lambda^2}C'\leq\frac{2}{\lambda}\mathcal I_L(0)+\frac{2}{\lambda^2}C'\leq\frac{2}{\lambda^2}C'.
		\end{equation}
		Therefore, there is a subsequence still labeled by $\{\phi_k\}$ such that
		\begin{equation}\label{phi_L^2}
		\phi_k\rightharpoonup\phi\quad\text{in $\mathcal H_L$}\quad\text{and}\quad \phi_k\rightarrow\phi\text{ in } L^2(\Omega_L).
		\end{equation}
		The straightforward computations show that
		$F_{\epsilon}(|{\bf{p}}|^2)$ is a convex function with respect to ${\bf {p}}\in\mathbb{R}^3$. It follows from \cite[Theorem 8.1]{Evans2010Partial}  that 
		\begin{equation}
		\int_{{\Omega}_L}F_\epsilon(|\nabla\phi|^2)\leq\lim_{k\rightarrow\infty}\int_{{\Omega}_L}F_\epsilon(|\nabla\phi_k|^2).
		\end{equation}
		On the other hand, similar to \eqref{S_Lestimate}, the following estimates hold,
		\begin{equation}
		\begin{split}
		\int_{\Sigma_L}(\phi_k-\phi)^2dx&\leq C(L)\int_{{\Omega}_L}\bigg|\nabla\big((\phi_k-\phi)^2\big)\bigg|dx\\[2mm]
		&\leq C(L)\int_{{\Omega}_L}|\phi-\phi_k|\cdot|\nabla\phi_k-\nabla\phi|dx\\[2mm]
		&\leq C(L)\bigg(\int_{{\Omega}_L}|\phi-\phi_k|^2dx\bigg)^{\frac{1}{2}}\bigg(\int_{{\Omega}_L}|\nabla\phi-\nabla\phi_k|^2dx\bigg)^{\frac{1}{2}},
		\end{split}
		\end{equation}
		where $C(L)$ is a constant depending on $L$.
		This, together with \eqref{phi_L^2}, yields
		\begin{equation}
		\lim_{k\rightarrow\infty}\int_{\Sigma_{L}}|\phi_k-\phi|dx'=0.
		\end{equation}
		Therefore,
		\begin{equation}
		\mathcal I_L(\phi)\leq\lim_{k\rightarrow\infty} \mathcal I_L(\phi_k).
		\end{equation}
		Hence $\mathcal I_L$ achieves its minimum at $\phi\in \mathcal H_L$. Furthermore, one has
		\begin{equation}
		\begin{split}
		\|\nabla\phi\|_{L^2(\Omega_L)}^2&\leq\frac{1}{\lambda}\int_{{\Omega}_L}F_\epsilon(|\nabla\phi|^2)dx=\frac{1}{\lambda}\bigg(\mathcal I_L(\phi)+\frac{m_0}{|\Sigma_{L}|}\int_{\Sigma_L}\phi dx'\bigg)\\ &\leq\frac{1}{\lambda}\bigg( \mathcal I_L(0)+\frac{m_0}{|\Sigma_{L}|}\int_{\Sigma_L}\phi dx'\bigg)
		\leq C \frac{m_0}{|\Sigma_{L}|}|{\Omega}_L|^{\frac{1}{2}}\|\nabla\phi\|_{L^2({\Omega}_L)}.
		\end{split}
		\end{equation}
		This implies
		\begin{equation}
		\|\nabla\phi\|_{L^2(\Omega_L)}^2\leq C\frac{m_0^2}{|\Sigma_{L}|^2} |{\Omega}_L|.
		\end{equation}
		Therefore, we have
		\begin{equation}
		\frac{1}{|{\Omega}_L|}\int_{{\Omega}_L}|\nabla\phi|^2dx\leq C\frac{m_0^2}{|\Sigma_{L}|^2}\leq C\frac{m_0^2}{\underline{\text{S}}^2}.
		\end{equation}
		Finally, given any $t\in\mathbb{R}$ and any  function $\psi\in \mathcal H_L$, obviously $\phi+t\psi\in \mathcal H_L$.
		Let $\sigma(t)=\mathcal I_L(\phi+t\psi)$.
		 Since $\sigma(t)$ achieves its minimum at $t=0$, one has $\sigma^\prime(0)=0.$ The straightforward computations give
		\begin{equation}
		0=\sigma'(0)=\int_{{\Omega}_L}H_\epsilon(|\nabla\phi|^2)\nabla\phi\cdot\nabla\psi dx-\frac{m_0}{|\Sigma_{L}|}\int_{\Sigma_{L}}\psi dx'.
		\end{equation}
		This means that $\phi$ is a weak solution of \eqref{truncated_equation}. Hence the proof of the lemma is completed.
		\end{proof}

	From now on, denote  $\Omega(t_1,t_2)={\Omega}\cap\{t_1<x_3<t_2\}$. With Lemma \ref{lemma1}, similar to the proof for  \cite[Proposition 4]{Du2011}, we have the following two propositions.
	\begin{proposition}\label{local_proposition}
		For given $t_1<t_2$, let
			\begin{equation*}
		l_1=\dashint_{\Omega(t_1-1,t_1)}\phi dx\quad\text{and}\quad  l_2=\dashint_{\Omega(t_2,t_2+1)}\phi dx.
		\end{equation*}
		It holds that	\begin{equation}\label{local property}
		|l_2-l_1|\leq C\int_{\Omega(t_1-1,t_2+1)}|\nabla\phi|dx,
		\end{equation}
			where $C$ is a positive constant depending on $\Omega$, independent of $t_1$ and $t_2$.
		\end{proposition}

	\begin{proposition}\label{local_proposition2}
		Let $\phi$ be a solution of the problem \eqref{truncated_equation}. For any $t\in\big(\frac{-L}{4},\frac{L}{4}\big)$, one has
		\begin{equation}\label{local_estimate}
		\dashint_{\Omega(t,t+1)}|\nabla\phi|^2dx<Cm_0^2.
		\end{equation}
		where $C$ is not dependent on $t$.
 	\end{proposition}

Since $\phi$ is a weak solution of a quasilinear elliptic equation of divergence form, similar to \cite[Lemmas 6 and 7]{Du2011}, using the Nash-Moser iteration yields that there exists a positive constant $N'<\frac{L}{2}$ such that
\begin{equation}
\|\nabla\phi\|_{L^\infty(\Omega(-K',K'))}\leq Cm_0\quad\text{and}\quad \|\nabla\phi\|_{C^{0,\alpha}(\Omega(-K',K'))}\leq Cm_0\quad \text{for $K'\in(N',\frac{L}{2})$}.
\end{equation}
\subsection{The existence and uniqueness of the subsonic flows}\label{existence}

Now, we are in position to prove the existence and uniqueness of the subsonic solution in the whole domain $\Omega$ and remove the coefficients truncations \eqref{densitytruncation} for the equation \eqref{truncated_equation}.
\begin{lemma}\label{thm1} There exists a critical value $\hat m>0$ such that if $m_0<\hat m$, then there exists a unique subsonic solution of \eqref{Potentialequation}. Moreover, $\max\limits_{x\in\Omega}|\nabla\phi|$ is a continuous function of $m_0$.
	\end{lemma}
\begin{proof}
For any  fixed $\hat{x}\in{\Omega}$, choose $L$ large enough such that $\hat{x}\in{\Omega}(\frac{-L}{4},\frac{L}{4})$.
Let $\phi_L$ be the solution of the truncated problem \eqref{truncated_equation}, and denote $\hat\phi_L=\phi_{L}-\phi_L(\hat{x})$. Obviously, $\hat\phi_L$ also satisfies \eqref{truncated_equation}, then
\begin{equation}\label{est67}
\|\nabla\hat\phi_L\|_{L^\infty({\Omega(-K',K')})}\leq Cm_0\quad \text{and} \quad \|\nabla\hat\phi_L\|_{C^{0,\alpha}({\Omega}(-K',K'))}\leq Cm_0.
\end{equation}
Therefore, by the diagonal procedure, there exists a subsequence $\{\hat\phi_{L_n}\}$ and a function $\phi\in C^{1,\alpha}_{loc}({\Omega})$ such that
\begin{equation*}
\lim_{n\rightarrow\infty}\|\hat\phi_{L_n}-\phi\|_{C^{1,\delta}({\Omega}{(-K',K')})}=0 \quad\text{with $\delta<\alpha$}.
\end{equation*}
Furthermore,
 $\phi$ is a strong solution to
\begin{equation}
\begin{cases}
\big(H_\epsilon(|\nabla\phi|^2)\delta_{ij}+2H_\epsilon'(|\nabla\phi|^2)\partial_i\phi\partial_j\phi\big)\partial_{ij}^2\phi=0&\text{ in }{\Omega},\\
\frac{\partial\phi}{\partial\textbf{n}}=0 &\text{ on }\partial\Omega.
\end{cases}
\end{equation}
Similar to the proof of \cite[Lemmas 6 and 7]{ Du2011}, one gets
\begin{equation} \phi\in C^{2,\alpha}_{loc}({\Omega})\quad\text{and}\quad
\|\nabla\phi\|_{C^{1,\alpha}({\Omega})}\leq Cm_0.
\end{equation}
Choosing $m_0$ small enough such that $Cm_0\leq1-\epsilon$, then we have
\begin{equation*}
H_\epsilon(|\nabla\phi|^2)=\rho(|\nabla\phi|^2).
\end{equation*}
Hence, $\phi$ indeed solves the problem \eqref{Potentialequation}.

Next, for  the uniqueness of the uniformly subsonic solution, one may refer to \cite{Du2011} for the proof.

Finally, we show that $\max\limits_{\Omega}|\nabla\phi|$ depends on $m_0$ continuously. Let $\{m_j\}$ be a sequence satisfying $m_j \uparrow m$ and $\phi^j$ be the unique subsonic solution of \eqref{Potentialequation} with mass flux $m_j$. Then the Areza-Ascoli Theorem leads to that for some $\alpha'<\alpha$, one has
\begin{equation}
\nabla\phi^j\rightarrow\nabla\phi_0\quad\text{in } C^{0,\alpha'}(\Omega),
\end{equation}
where $\phi_0$ is the solution of \eqref{Potentialequation} with mass flux $m$. One can conclude that for this convergence $\max\limits_{x\in\Omega}|\nabla\phi|$ is a continuous function of $m_0$.

	Let $\{r_i\}_{i=1}^{\infty}$ be a strictly increasing sequence satisfying $\lim\limits_{i\rightarrow\infty}r_i=1$. Because of the continuity of $Q(m)$,  there exists the largest $R_n>0$ such that
	\begin{equation}
	0<Q(m)<r_n\text{ for any } m\in(0,R_n).
	\end{equation}
	Obviously $R_{n+1}\geq R_n$.
	Moreover, 
	\begin{equation}
	R_n=\int_{{\Sigma_t}}\rho(|\nabla\phi|^2)\frac{\partial\phi}{\partial\textbf{l}}ds\leq|{{\Sigma_t}}|\rho(Q^2(R_n))Q(R_n).
	\end{equation}
	Hence $\{R_n\}$ is bounded. Set $\lim\limits_{n\rightarrow\infty}R_n=\hat{m}$. Therefore, for any $m_0<\hat{m}$, there exists an $n$ such that $m_0<R_n$, $Q(m_0)<r_n<1$.
	Moreover, for any $\bar Q\in(0,1)$, there exists an $n$ such that $\bar Q\in(0,r_n)$. Using the continuity of $Q(m)$ again yields that there exists an $m_0\in(0,R_n)$ satisfying $Q(m_0)=\bar Q$.
\end{proof}

Theorem \ref{theorem_1} is the direct consequence in Section \ref{existence}, Lemma \ref{thm1}.

\section{Convergence rates at far fields}

It follows from the study in \cite{Du2011} that the uniformly subsonic solution in Theorem \ref{theorem_1} tends to an uniform state at far fields if the nozzle tends to be a straight one.
In this section, we investigate the convergence rates of uniform subsonic flows at far fields and prove Theorem \ref{convergence_thm}.
Let $\phi_1(x)$ be the uniformly subsonic solution of
\begin{equation}\label{Potentialeq}
\begin{cases}
\Div(\rho(|\nabla\phi_1|^2)\nabla\phi_1)=0    &\text{in }\Omega, \\[2mm]
\frac{\partial\phi_1}{\partial \mathbf n}=0 &\text{on }\partial\Omega,\\[2mm]
\int_{\Sigma_t}\rho(|\nabla\phi_1|^2)\frac{\partial\phi_1}{\partial \mathbf l}ds=m_0
\end{cases}
\end{equation}
obtained in Theorem \ref{theorem_1} which satisfies
\begin{equation}\label{phi_property}
\|\nabla\phi_1\|_{C^{1,\alpha}(\Omega)}\leq Cm_0.
\end{equation}

 The basic idea is to establish the local energy decay via weighted energy estimates, which is the core part to get the convergence rates. The pointwise convergence rates is proved by the Nash-Moser iteration. The whole proof is divided into three sections. We start with the simple case where the nozzle boundary is straight when $x_3>K$.\\
\subsection{Energy estimates for the boundary is straight at far fields}\label{3.1}
Assume $f_1=\bar f$ for $x_3>K$ with some positive constant $K$. In this case, one has $n_3=0$ for $\partial\Omega\cap\{x_3>K\}$. Let $\phi_1$ be the solution of \eqref{Potentialeq} and  $\phi_2=\bar qx_3$. Obviously, $\phi_2$ satisfies
\begin{equation}
\begin{cases}
\Div(\rho(|\nabla\phi_2|^2)\nabla\phi_2)=0 & \text{in } \Omega\cap \{x_3>K\}, \\
\frac{\partial\phi_2}{\partial\mathbf n}=0& \text{on}\ \partial\Omega\cap \{x_3>K\}.
\end{cases}
\end{equation}
Denote $\Phi=\phi_1-\phi_2$. Then $\Phi$ satisfies
\begin{equation}\label{uniquenesseq}
\begin{cases}
\partial_i(\mathfrak a_{ij}\partial_j\Phi)=0&\quad \text{in } \Omega\cap \{x_3>K\},\\
\frac{\partial\Phi}{\partial{\bf n}}=0&\quad \text{on}\ \partial\Omega\cap \{x_3>K\},
\end{cases}
\end{equation}
where 
\begin{equation}\label{convergence_a_coefficient}
\mathfrak a_{ij}=\int_0^1 \rho(\hat{q}^2)\delta_{ij}+2\rho'(\hat{q}^2)(s\partial_j\phi_1+(1-s)\partial_j\phi_2)(s\partial_i\phi_1+(1-s)\partial_i\phi_2)ds
\end{equation}
with
\begin{equation}\label{convergence_q_coefficient}
\hat{q}^2=|s\nabla\phi_1+(1-s)\nabla\phi_2|^2.
\end{equation}

The straightforward computations show that there exist two constants $\lambda$ and $\Lambda$ such that
\begin{equation}\label{aij_elliptic}
\lambda|\xi|^2\leq\mathfrak a_{ij}\xi_i\xi_j\leq\Lambda|\xi|^2\quad\text{for $\xi\in\mathbb{R}^3$}.
\end{equation}
Moreover, one can increase $\Lambda$ so that the following Poincar$\acute{e}$ inequality holds on each cross section,
\begin{equation}\label{Poincare_inequality}
\bigg\|\mathcal Z-\dashint_{{\Sigma}_t}\mathcal Z ds\bigg\|_{L^2({\Sigma}_t)}\leq\Lambda\|\nabla \mathcal Z\|_{L^2({\Sigma}_t)}\quad\text{for any $\mathcal Z\in H^2_{loc}(\Omega(t-\epsilon,t+\epsilon))$.}
\end{equation}

For any $t_1<t_2$, $h$ and $\beta$ are constants to be determined later. Denote
\begin{equation}\label{zeta}
\zeta(x_3;t_1,t_2,,\beta,h)=
\begin{cases}
1  &  x_3\leq t_1-h,\\
e^{\beta(x_3-t_1+h)} &   t_1-h< x_3\leq t_1,\\
e^{\beta h} & t_1< x_3\leq t_2,\\
e^{\beta h}\cdot e^{-\beta(x_3-t_2)}  &  t_2<x_3\leq t_2+h,\\
1 & x_3 >t_2+h.
\end{cases}
\end{equation}
  Multiplying $\Phi(\zeta(x_3,t_1,t_2,\beta,h)-1)$ on both sides of \eqref{uniquenesseq} and taking integral on $\Omega(t_1-h,t_2+h)$ yield
\begin{equation}
\begin{split}
&\quad\int_{\Omega(t_1-h,t_2+h)}\mathfrak a_{ij}\partial_i\Phi\partial_j\Phi(\zeta-1)dx+\int_{\Omega(t_1-h,t_2+h)}\mathfrak a_{i3}\partial_i\Phi\Phi\partial_3\zeta dx\\&=\int_{\partial\Omega\cap\overline{\Omega(t_1-h,t_2+h)}}\mathfrak a_{ij}\partial_j\Phi\Phi(\zeta-1)n_ids.
\end{split}
\end{equation}  
For the boundary term, one has
\begin{equation}\label{unique boundary}
\mathfrak a_{ij}\partial_j\Phi n_i=\big(\rho(|\nabla\phi_1|^2)\partial_i\phi_1-\rho(|\nabla\phi_2|^2)\partial_i\phi_2\big)\cdot n_i=0.
\end{equation}
Moreover, the conserved mass flux on each cross section implies 
\begin{equation}\label{unique cross section}
\int_{\Sigma_t}\mathfrak a_{i3}\partial_i\Phi dx'=\int_{\Sigma_t}\rho(|\nabla\phi_1|^2)\partial_3\phi_1-\rho(|\nabla\phi_2|^2)\partial_3\phi_2 dx'=0.
\end{equation}
Set $\tilde\eta(t)=\int_{\Sigma_t}\Phi dx'$.
Combining \eqref{unique boundary} and \eqref{unique cross section} yields that
\begin{equation}\begin{split}
&\quad\lambda\int_{\Omega(t_1-h,t_2+h)}|\nabla\Phi|^2(\zeta-1)dx\\
&\leq-\int\limits_{[t_1-h,t_1]\cup[t_2,t_2+h]}\bigg(\frac{\tilde\eta(x_3)}{|\Sigma_{x_3}|}\partial_3\zeta\int_{\Sigma_{x_3}}\mathfrak a_{i3}\partial_i\Phi dx'\bigg)dx_3-\int_{\Omega(t_1-h,t_1)\cup\Omega(t_2,t_2+h)}\mathfrak a_{i3}\partial_i\Phi\bigg(\Phi-\frac{\tilde\eta(x_3)}{|\Sigma_{x_3}|}\bigg)\partial_3\zeta dx\\
&\leq \bigg(\int_{\Omega(t_1-h,t_1)\cup\Omega(t_2,t_2+h)}\bigg(\Phi-\frac{\tilde\eta(x_3)}{|\Sigma_{x_3}|}\bigg)^2(\partial_3\zeta)^2\zeta^{-1}dx\bigg)^{\frac{1}{2}}\bigg(\int_{\Omega(t_1-h,t_1)\cup\Omega(t_2,t_2+h)}(\mathfrak a_{i3}\partial_i\Phi)^2\zeta dx\bigg)^{\frac{1}{2}}.
\end{split}\end{equation}
It follows from \eqref{Poincare_inequality} that 
\begin{equation}\begin{split}
&\quad\bigg(\int_{\Omega(t_1-h,t_1)\cup\Omega(t_2,t_2+h)}\bigg(\Phi-\frac{\tilde\eta(x_3)}{|\Sigma_{x_3}|}\bigg)^2(\partial_3\zeta)^2\zeta^{-1}dx\bigg)^{\frac{1}{2}}\\
&=\bigg\{\int_{t_1-h}^{t_1}+\int_{t_2}^{t_2+h}(\partial_3\zeta)^2\zeta^{-1}\bigg[\int_{\Sigma_{x_3}}\bigg(\Phi-\frac{\tilde\eta(x_3)}{|\Sigma_{x_3}|}\bigg)^2dx'\bigg]dx_3\bigg\}^\frac{1}{2}\\
&\leq\bigg[\int_{t_1-h}^{t_1}+\int_{t_2}^{t_2+h}\Lambda^2(\partial_3\zeta)^2\zeta^{-1}\bigg(\int_{\Sigma_{x_3}}|\nabla\Phi|^2dx'\bigg)dx_3\bigg]^\frac{1}{2}\\
&\leq\bigg(\int_{\Omega(t_1-h,t_1)\cup\Omega(t_2,t_2+h)}\Lambda^2(\partial_3\zeta)^2\zeta^{-1}|\nabla\Phi|^2dx\bigg)^{\frac{1}{2}}.
\end{split}\end{equation}
Note that $\partial_3\zeta=\beta\zeta$ for $x_3\in[t_1-h,t_1]\cup[t_2,t_2+h]$, then
\begin{equation}\label{Phit1t2}
\lambda\int_{\Omega(t_1-h,t_2+h)}|\nabla\Phi|^2(\zeta-1)dx\leq\Lambda^2\beta\int_{\Omega(t_1-h,t_1)\cup\Omega(t_2,t_2+h)}|\nabla\Phi|^2\zeta dx.
\end{equation}
Set $\beta=\frac{\lambda}{\Lambda^2 }$, then we have the following estimate
\begin{equation}
e^{\beta h}\int_{\Omega(t_1,t_2)}|\nabla\Phi|^2dx\leq\int_{\Omega(t_1-h,t_2+h)}|\nabla\Phi|^2dx.
\end{equation}
Taking $t_1=T$, $t_2=T+1$ and $h=\frac{T}{2}$ yields
\begin{equation}\label{cylinder_energy_estimate} e^{\beta T}\int_{\Omega(T,T+1)}|\nabla\Phi|^2dx\leq  \int_{\Omega(\frac{T}{2},\frac{3T}{2}+1)} |\nabla\Phi|^2dx\leq C(T+1).
\end{equation}
Thus, there must be a positive constant $\bar\alpha$ such that
\begin{equation}\label{cylinder_L2}
\int_{\Omega(T,T+1)}|\nabla\Psi|^2dx\leq\frac{1}{e^{\bar\alpha T}}.
\end{equation}
\subsection{ Energy estimates for the boundary has the algebraic convergence}\label{3.2} When the boundary satisfies \eqref{boundary_algebratic_rate}, the unit outer normal direction of the boundary can be written as
\begin{equation*}
{\bf{n}}=(n_1,n_2,n_3)=\frac{1}{\sqrt{G}}\bigg(\cos\theta+\frac{\partial f_1}{\partial\theta}\frac{\sin\theta}{r}, \sin\theta-\frac{\partial f_1}{\partial\theta}\frac{\cos\theta}{r},-\frac{\partial f_1}{\partial x_3}\bigg),
\end{equation*}
where \begin{equation*}G=1+\bigg(\frac{\partial f_1}{\partial \theta}\bigg)^2\frac{1}{r^2}+\bigg(\frac{\partial f_1}{\partial x_3}\bigg)^2.\end{equation*}
Let $\phi_1$ be the solution of \eqref{Potentialeq} and  $\phi_2=\bar qx_3$. 
 Obviously, $\phi_2$ satisfies
\begin{equation}
\begin{cases}
\Div(\rho(|\nabla\phi_2|^2)\nabla\phi_2)=0 & \text{in } \Omega\cap \{x_3>K\}, \\
\frac{\partial\phi_2}{\partial\mathbf n}=\bar qn_3& \text{on}\ \partial\Omega\cap \{x_3>K\}.
\end{cases}
\end{equation}

Denote $\Psi=\phi_1-\phi_2$. It is easy to check that $\Psi$ satisfies
\begin{equation}\label{cylinderPsieq}
\begin{cases}
\partial_i(a_{ij}\partial_j\Psi)=0 &\text{in} \ \Omega\cap \{x_3>K\}, \\
\frac{\partial\Psi}{\partial \mathbf n}=-\bar q n_3  &\text{on}\ \partial\Omega\cap \{x_3>K\},
\end{cases}
\end{equation}
where $a_{ij}$ is same as in $\mathfrak a_{ij}$
and satisfies \eqref{aij_elliptic}.
In fact, $\Psi$ also satisfies
\begin{equation}\label{est103.5}
 a_{ij}\partial_j\Psi n_i=-\rho(\bar q^2)\bar q n_3\quad\text{on}\ \partial\Omega\cap \{x_3>K\}.
 \end{equation}
On the boundary $\partial\Omega\cap\{x_3>K\}$, it is easy to check that
\begin{equation}\label{bdyterm}
|a_{ij}\partial_j\Psi n_i|=\big|(\rho(|\nabla\phi_1|^2)\partial_i\phi_1-\rho(|\nabla\phi_2|^2)\partial_i\phi_2) n_i\big|=\big|\rho(\bar q^2)\bar qn_3\big|\leq\frac{C}{x_3^{l+1}}.
\end{equation}

On the cross section $\Sigma_{x_3}$, one has
\begin{equation}\label{generalcrosssection}
\begin{split}
&\quad\bigg|\int_{\Sigma_{x_3}}a_{i3}\partial_i\Psi dx'\bigg|=\bigg|\int_{\Sigma_{x_3}}\rho(|\nabla\phi_1|^2)\partial_3\phi_1-\rho(|\nabla\phi_2|^2)\partial_3\phi_2dx'\bigg|\\
&=\bigg|m_0-\rho(\bar q^2)\bar q|\Sigma_{x_3}|\bigg|=\bigg|\rho(\bar q^2)\bar q(|\overline\Sigma|-|\Sigma_{x_3}|)\bigg|\\
&=\bigg|\rho(\bar q^2)\bar q\pi (f_1^2-\bar f^2)\bigg|\leq \frac{C}{x_3^l}.
\end{split}
\end{equation}

Let $\bar K$ be a positive integer to be determined later. Choose $t_1=T$ and $t_2=t_1+\bar  K$.
Denote
\begin{equation*}
 s_1=\dashint_{\Omega(t_1-1,t_1)}\Psi dx\quad \text{and}\quad s_2=\dashint_{\Omega(t_2,t_2+1)}\Psi dx.
\end{equation*}
Define \begin{equation*}
\hat{\Psi}(x;t_1,t_2,s_1,s_2)=\left\{
\begin{array}{llr}
\Psi(x)-s_1,  & {x_3<t_1,}\\
\Psi(x)-s_1-\frac{s_2-s_1}{t_2-t_1}(x_3-t_1), & {t_1\leq x_3\leq t_2,}\\
\Psi(x)-s_2,  & {x_3>t_2}.
\end{array} \right.
\end{equation*}
Let $\zeta$ be defined in \eqref{zeta} and $h=1$. Multiplying $\hat\Psi\big(\zeta(x_3;t_1,t_2,\hat\beta,1)-1\big)$ on both sides of the equation \eqref{cylinderPsieq} and integrating on $\Omega(t_1-1,t_2+1)$ yield
\begin{equation}\begin{split}
&\quad\int_{\Omega(t_1-1,t_2+1)}a_{ij}\partial_i\Psi\partial_j\Psi(\zeta-1)dx+\int_{\Omega(t_1,t_2)}a_{i3}\partial_i\Psi\frac{s_1-s_2}{t_2-t_1}(\zeta-1 )dx\\&=-\int_{\Omega(t_1-1,t_2+1)}a_{i3}\partial_i\Psi\hat\Psi\zeta_{x_3}dx+\int_{\partial\Omega\cap\overline{\Omega(t_1-1,t_2+1)}}\hat\Psi(\zeta-1)a_{ij}\partial_j\Psi n_ids,
\end{split}
\end{equation}
where $\zeta-1=0$ at $x_3=t_1-h$ and $x_3=t_2+h$ is used.
Thus
\begin{equation*}
\begin{split}
&\lambda\int_{\Omega(t_1-1,t_2+1)}(\zeta-1)|\nabla\Psi|^2dx\leq\bigg|\frac{s_2-s_1}{t_2-t_1}\int_{\Omega(t_1,t_2)}a_{i3}\partial_i\Psi(\zeta-1 )dx\bigg|\\[2mm] \quad&-\int_{\Omega(t_1-1,t_1)\cup\Omega(t_2,t_2+1)}a_{i3}\partial_i\Psi\hat\Psi\zeta_{x_3}dx+\int_{\partial\Omega\cap\overline{\Omega(t_1-1,t_2+1)}}\hat\Psi(\zeta-1)a_{ij}\partial_j\Psi n_ids.
\end{split}
\end{equation*}

Set \begin{equation*}\eta_1(x_3)=\int_{\Sigma_{x_3}}(\Psi-s_1) dx'\quad\text{and}\quad \eta_2(x_3)=\int_{\Sigma_{x_3}}(\Psi-s_2) dx'.\end{equation*} One has
\begin{equation}\label{I_i}\begin{split}
&\quad\lambda\int_{\Omega(t_1-1,t_2+1)}(\zeta-1)|\nabla\Psi|^2dx\\&\leq-\int_{\Omega(t_1-1,t_1)}a_{i3}\partial_i\Psi\bigg(\hat\Psi-\frac{\eta_1(x_3)}{|\Sigma_{x_3}|}\bigg)\zeta_{x_3}dx-\int_{\Omega(t_2,t_2+1)}a_{i3}\partial_i\Psi\bigg(\hat\Psi-\frac{\eta_2(x_3)}{|\Sigma_{x_3}|}\bigg)\zeta_{x_3}dx\\[2mm] &\quad+\bigg|\frac{s_2-s_1}{t_2-t_1}\int_{\Omega(t_1,t_2)}a_{i3}\partial_i\Psi(\zeta-1 )dx\bigg|-\int_{\Omega(t_1-1,t_1)}a_{i3}\partial_i\Psi\frac{\eta_1(x_3)}{|\Sigma_{x_3}|}\zeta_{x_3}dx\\[2mm]
&\quad-\int_{\Omega(t_2,t_2+1)}a_{i3}\partial_i\Psi\frac{\eta_2(x_3)}{|\Sigma_{x_3}|}\zeta_{x_3}dx+\int_{\partial\Omega\cap\overline{\Omega(t_1-1,t_2+1)}}\hat\Psi(\zeta-1)a_{ij}\partial_j\Psi n_ids=\sum\limits_{i=1}^6 I_i.
\end{split}
\end{equation}

We estimate $I_i\  (i=1,\cdots,6)$ one by one.
Applying H$\ddot{\text{o}}$lder inequality to $I_1$ gives
\begin{equation}\begin{split}
|I_1|&\leq\bigg{[}\int_{\Omega(t_1-1,t_1)}(a_{i3}\partial_i\Psi)^2\zeta dx\bigg{]}^\frac{1}{2}
\bigg{[}\int_{\Omega(t_1-1,t_1)}\bigg(\hat\Psi-\frac{\eta_1(x_3)}{|\Sigma_{x_3}|}\bigg)^2\zeta_{x_3}^{2}\zeta^{-1} dx\bigg{]}^\frac{1}{2}\\
&\leq\bigg{[}\int_{\Omega(t_1-1,t_1)}(a_{i3}\partial_i\Psi)^2\zeta dx\bigg{]}^\frac{1}{2}\bigg{[}\int_{t_1-1}^{t_1}\int_{\Sigma_{x_3}}\bigg(\hat\Psi-\frac{\eta_1(x_3)}{|\Sigma_{x_3}|}\bigg)^2dx'\zeta_{x_3}^2\zeta^{-1}dx_3\bigg{]}^{\frac{1}{2}}\\
&\leq\bigg{[}\int_{\Omega(t_1-1,t_1)}\Lambda^2|\nabla\Psi|^2\zeta dx\bigg{]}^\frac{1}{2}\bigg{[}\int_{\Omega(t_1-1,t_1)}\Lambda^2\zeta_{x_3}^2\zeta^{-1}|\nabla\hat\Psi|^2dx\bigg{]}^\frac{1}{2}\\
&\leq\Lambda^2\hat\beta\int_{\Omega(t_1-1,t_1)}|\nabla\Psi|^2\zeta dx,
\end{split}\end{equation}
where the third inequality follows from the Poincar$\acute{e}$ inequality \eqref{Poincare_inequality}.
Similarly, one has
\begin{equation}
	|I_2|\leq\Lambda^2\hat\beta\int_{\Omega(t_2,t_2+1)}|\nabla\Psi|^2\zeta dx.
	\end{equation}
Taking $\hat\beta=\frac{\lambda}{\Lambda^2}$ yields
\begin{equation}
\begin{split}
\int_{\Omega(t_1-1,t_2+1)}|\nabla\Psi|^2(\zeta-1)dx&\leq\int_{\Omega(t_1-1,t_1)\cup\Omega(t_2,t_2+1)}|\nabla\Psi|^2\zeta dx+\frac{1}{\lambda}(I_3+I_4+I_5+I_6).
\end{split}
\end{equation}
It follows from the definition of $\zeta$ that 
\begin{equation}
\begin{split}
 e^{\hat\beta }\int_{\Omega(t_1,t_2)}|\nabla\Psi|^2dx\leq\int_{\Omega(t_1-1,t_2+1)}|\nabla\Psi|^2dx+\frac{1}{\lambda}(I_3+I_4+I_5+I_6).
\end{split}
\end{equation}

As same as the proof for \eqref{local property}, we can also prove
\begin{equation}\label{localproperty2}
|s_2-s_1|\leq C\int_{\Omega(t_1-1,t_2+1)}|\nabla\Psi|dx.
\end{equation}  Using \eqref{generalcrosssection} and \eqref{localproperty2} gives
\begin{equation}\label{I_1}\begin{split}
\quad \frac{1}{\lambda}|I_3|&\leq\frac{|s_2-s_1|(e^{\hat\beta}-1)}{\lambda(t_2-t_1)}\int_{t_1}^{t_2}\bigg|\int_{\Sigma_{x_3}}a_{i3}\partial_i\Psi dx'\bigg|dx_3\\
&\leq\frac{|s_2-s_1|(e^{\hat\beta}-1)}{\lambda(t_2-t_1)}\int_{t_1}^{t_2}\frac{C}{x_3^l} dx_3\\&\leq\frac{C (e^{\hat\beta}-1)|s_2-s_1|}{\lambda(t_1)^l}
\leq \frac{C (e^{\hat\beta}-1)}{(t_1)^l}\int_{\Omega(t_1-1,t_2+1)}|\nabla\Psi|dx\quad \\[2mm]
&\leq\frac{C(e^{\hat\beta}-1)}{(t_1)^l}|\Omega(t_1-1,t_2+1)|^{\frac{1}{2}}\bigg(\int_{\Omega(t_1-1,t_2+1)}|\nabla\Psi|^2dx\bigg)^\frac{1}{2}\\[2mm]
&\leq \frac{C(e^{\hat\beta}-1)^2(t_2-t_1+2)}{(t_1)^{2l}\epsilon}+\epsilon\int_{\Omega(t_1-1,t_2+1)}|\nabla\Psi|^2dx.\end{split}
\end{equation}

 Now, we estimate $I_4$ as follows
\begin{equation}\label{I_2}
\begin{split}
\quad\frac{1}{\lambda}|I_4|&\leq\bigg|\frac{1}{\lambda}\int_{\Omega(t_1-1,t_1)}a_{i3}\partial_i\Psi\frac{\eta_1(x_3)}{|\Sigma_{x_3}|}\zeta_{x_3}dx\bigg|\\[2mm]
&\leq\bigg|\frac{1}{\lambda}\int_{t_1-1}^{t_1}\bigg[\bigg(\int_{\Sigma_{x_3}}a_{i3}\partial_i\Psi dx_1dx_2\bigg)\frac{\eta_1(x_3)}{|\Sigma_{x_3}|}\zeta_{x_3}\bigg] dx_3\bigg|\\[2mm]
&\leq\frac{1}{\lambda}\bigg[\int_{t_1-1}^{t_1}\bigg(\int_{\Sigma_{x_3}}a_{i3}\partial_i\Psi dx_1dx_2\frac{\zeta_{x_3}}{|\Sigma_{x_3}|}\bigg)^2dx_3\bigg]^{\frac{1}{2}}\bigg[\int_{t_1-1}^{t_1}\big(\eta_1(x_3)\big)^2dx_3\bigg]^{\frac{1}{2}}\\[2mm]
&\leq\frac{1}{\lambda}\frac{C}{(t_1-1)^l}\bigg[\int_{t_1-1}^{t_1}\zeta^2_{x_3}dx_3\bigg]^{\frac{1}{2}}\bigg[\int_{t_1-1}^{t_1}\bigg(\int_{\Sigma_{x_3}}\big(\Psi(x)-s_1\big)dx'\bigg)^2dx_3\bigg]^{\frac{1}{2}}\\[2mm]
&\leq\frac{1}{\lambda}\frac{C}{(t_1-1)^l}\bigg(\int_{t_1-1}^{t_1}\hat\beta^2e^{2\hat\beta(x_3-t_1+1)}dx_3\bigg)^{\frac{1}{2}}\bigg[\int_{t_1-1}^{t_1}\bigg(\int_{\Sigma_{x_3}}\big(\Psi(x)-s_1\big)dx'\bigg)^2dx_3\bigg]^{\frac{1}{2}}\\
&\leq\frac{C}{(t_1-1)^l}\bigg(\frac{\hat\beta}{2}\big(e^{2\hat\beta}-1\big)\bigg)^\frac{1}{2}\bigg[\int_{t_1-1}^{t_1}\bigg(\int_{\Sigma_{x_3}}\big(\Psi(x)-s_1\big)^2dx'\int_{\Sigma_{x_3}}1dx'\bigg)dx_3\bigg]^{\frac{1}{2}}\\[2mm]
&\leq \frac{C}{(t_1-1)^l}\bigg(\frac{\hat\beta}{2}\big(e^{2\hat\beta}-1\big)\bigg)^\frac{1}{2}\bigg(\int_{\Omega(t_1-1,t_1)}|\nabla\Psi|^2dx\bigg)^\frac{1}{2}\\[2mm]
&\leq \frac{C\hat\beta(e^{2\hat\beta}-1)}{4\epsilon(t_1-1)^{2l}}+\frac{\epsilon}{2}\int_{\Omega(t_1-1,t_1)}|\nabla\Psi|^2dx,
\end{split}
\end{equation}
where the estimate \eqref{generalcrosssection} was used.
Similarly, one has
\begin{equation}\label{I_4}
\frac{1}{\lambda}|I_5|\leq\frac{C\hat\beta(e^{2\hat\beta}-1)}{4\epsilon(t_2)^{2l}}+\frac{\epsilon}{2}\int_{\Omega(t_2,t_2+1)}|\nabla\Psi|^2dx.
\end{equation}

For the boundary term $I_6$, it follows from \eqref{bdyterm}  that
\begin{equation}\label{bdyI2}
\begin{split}
\frac{1}{\lambda}|I_6|&\leq \frac{Ce^{\hat\beta}}{(t_1-1)^{l+1}}\int_{\partial\Omega\cap\overline{\Omega(t_1-1,t_2+1)}}|\hat\Psi| ds\\
&\leq \frac{Ce^{\hat\beta}}{(t_1-1)^{l+1}}
\sum\limits_{i=0}^{\bar K+1}\int_{\partial\Omega\cap\overline{\Omega(t_1-1+i,t_1+i)}}|\hat\Psi| ds\\
&\leq\frac{Ce^{\hat\beta}}{(t_1-1)^{l+1}}\sum\limits_{i=0}^{\bar K+1}\bigg(\int_{\Omega(t_1-1+i,t_1+i)}|\nabla\hat\Psi| dx+\int_{\Omega(t_1-1+i,t_1+i)}|\hat\Psi| dx\bigg)\\
&\leq\frac{Ce^{\hat\beta}}{(t_1-1)^{l+1}}\bigg(\int_{\Omega(t_1-1,t_2+1)}|\nabla\Psi|dx+|s_2-s_1|+\int_{\Omega(t_1-1,t_2+1)}|\hat\Psi| dx\bigg)\\
&\leq\frac{Ce^{\hat\beta}}{(t_1-1)^{l+1}}\bigg(\int_{\Omega(t_1-1,t_2+1)}|\nabla\Psi|dx+\int_{\Omega(t_1-1,t_2+1)}|\hat\Psi| dx\bigg).
\end{split}
\end{equation}

Now the key issue is to estimate the second term on the right hand side of \eqref{bdyI2}.
 Define
\begin{equation*}
\mathcal{B}_j=\dashint_{\Omega(t_1+j,t_1+j+1)}\Psi dx, \quad j=0,1,\cdots, \bar K-1.
\end{equation*}
As same as the estimate \eqref{local property}, one has
\begin{equation}\label{mathbjs1}
\mathcal |\mathcal B_j-s_1|\leq C\int_{\Omega(t_1-1,t_1+j+1)}|\nabla\Psi|dx.
\end{equation}
It follows from \eqref{mathbjs1} that
\begin{equation}\label{hatpsit1t2}\begin{split}
&\quad\int_{\Omega(t_1,t_2)}|\Psi-s_1|dx\leq\sum\limits_{j=0}^{\bar K-1}\int_{\Omega(t_1+j,t_1+j+1)}|\Psi-\mathcal B_j|+|\mathcal B_j-s_1| dx\\
&\leq C\sum\limits_{j=0}^{\bar K-1}\int_{\Omega(t_1+j,t_1+j+1)}|\nabla\Psi|dx+ C\sum\limits_{j=0}^{\bar K-1}|\mathcal B_j-s_1|\\
&\leq C\int_{\Omega(t_1,t_2)}|\nabla\Psi|dx+\sum\limits_{j=0}^{\bar K-1}\int_{\Omega(t_1-1,t_2+1)}|\nabla\Psi|dx\\
&\leq C(t_2-t_1+2)\int_{\Omega(t_1-1,t_2+1)}|\nabla\Psi|dx.
\end{split}\end{equation}
Hence,
\begin{equation}\label{hatpsit1-1}
\begin{split}
&\quad\int_{\Omega(t_1-1,t_2+1)}|\hat\Psi|
dx=\int_{\Omega(t_1-1,t_1)}|\hat\Psi| dx+\int_{\Omega(t_2,t_2+1)}|\hat\Psi| dx+\int_{\Omega(t_1,t_2)}|\hat\Psi| dx\\[2mm]&
\leq C\bigg(\int_{\Omega(t_1-1,t_1)\cup\Omega(t_2,t_2+1)}|\nabla\Psi| dx\bigg)+\int_{\Omega(t_1,t_2)}\bigg|\Psi-s_1-\frac{s_2-s_1}{t_2-t_1}(x_3-t_1)\bigg|dx\\&
\leq C\bigg(\int_{\Omega(t_1-1,t_1)\cup\Omega(t_2,t_2+1)}|\nabla\Psi| dx\bigg)+|s_2-s_1|\int_{t_1}^{t_2}\int_{\Sigma_{x_3}}\frac{x_3-t_1}{t_2-t_1}dx'dx_3+\int_{\Omega(t_1,t_2)}|\Psi-s_1| dx\\
&\leq C\bigg(\int_{\Omega(t_1-1,t_1)\cup\Omega(t_2,t_2+1)}|\nabla\Psi| dx\bigg)+C(t_2-t_1)|s_2-s_1|+C(t_2-t_1+2)\int_{\Omega(t_1-1,t_2+1)}|\nabla\Psi|dx\\
&\leq C(t_2-t_1+2)\int_{\Omega(t_1-1,t_2+1)}|\nabla\Psi|dx.
\end{split}
\end{equation}
Therefore, it follows from \eqref{bdyI2} and  \eqref{hatpsit1-1} that
\begin{equation}\label{I_6}\begin{split}
\frac{1}{\lambda}|I_6|
&\leq \frac{Ce^{\hat\beta}(t_2-t_1+2)}{(t_1-1)^{l+1}}\int_{\Omega(t_1-1,t_2+1)}|\nabla\Psi|dx\\
&\leq \frac{Ce^{\hat\beta}(t_2-t_1+2)^{\frac{3}{2}}}{(t_1-1)^{l+1}} \bigg(\int_{\Omega(t_1-1,t_2+1)}|\nabla\Psi|^2dx\bigg)^\frac{1}{2}\\
&\leq  \frac{Ce^{2\hat\beta}(t_2-t_1+2)^3}{(t_1-1)^{2l+2}\epsilon}+\epsilon\int_{\Omega(t_1-1,t_2+1)}|\nabla\Psi|^2dx.
\end{split}
\end{equation}

Collecting \eqref{I_1}, \eqref{I_2}, \eqref{I_4} and \eqref{I_6} together gives
\begin{equation}
\begin{split}
&e^{\hat\beta}\int_{\Omega(t_1,t_2)}|\nabla\Psi|^2dx\leq\frac{C(e^{\hat\beta}-1)^2(t_2-t_1+2)}{(t_1)^{2l}\epsilon}+\epsilon\int_{\Omega(t_1-1,t_2+1)}|\nabla\Psi|^2dx+\frac{C\hat\beta(e^{2\hat\beta}-1)}{4\epsilon(t_1-1)^{2l}}\\[2mm]&\quad+\frac{\epsilon}{2}\int_{\Omega(t_1-1,t_1)}|\nabla\Psi|^2dx+\frac{C\hat\beta(e^{2\hat\beta}-1)}{4\epsilon(t_2)^{2l}}+\frac{\epsilon}{2}\int_{\Omega(t_2,t_2+1)}|\nabla\Psi|^2dx+\frac{Ce^{2\hat\beta}(t_2-t_1+2)^3}{(t_1-1)^{2l+2}\epsilon}\\&\quad+\epsilon\int_{\Omega(t_1-1,t_2+1)}|\nabla\Psi|^2dx+\int_{\Omega(t_1-1,t_2+1)}|\nabla\Psi|^2dx.
\end{split}
\end{equation}
Therefore, one has
\begin{equation}\label{estimate1}
\begin{split}
e^{\hat\beta}\int_{\Omega(t_1,t_2)}|\nabla\Psi|^2dx&\leq \frac{C(e^{\hat\beta}-1)^2(t_2-t_1+2)}{\epsilon(t_1)^{2l}}+\frac{C\hat\beta(e^{2\hat\beta}-1)}{\epsilon(t_1-1)^{2l}}+\frac{Ce^{2\hat\beta}(t_2-t_1+2)^3}{(t_1-1)^{2l+2}\epsilon}+\frac{C\hat\beta(e^{2\hat\beta}-1)}{\epsilon(t_2)^{2l}}\\&\quad+(\epsilon+1)\int_{\Omega(t_1-1,t_2+1)}|\nabla\Psi|^2dx.
\end{split}
\end{equation}
Choosing $\epsilon$ small enough such that $\frac{\epsilon+1}{e^{\hat\beta}}\leq\mathfrak{b}_0<1$, then
\begin{equation}
\int_{\Omega(t_1,t_2)}|\nabla\Psi|^2dx\leq \mathfrak b_0\int_{\Omega(t_1-1,t_2+1)}|\nabla\Psi|^2dx+\frac{C(t_2-t_1+2)}{(t_1-1)^{2l}}+\frac{C(t_2-t_1+2)^3}{(t_1-1)^{2l+2}},
\end{equation}
where $C$ is a constant depending on $\hat\beta$.
Thus,
\begin{equation}
\begin{split}
&\quad\frac{1}{t_2-t_1}\int_{\Omega(t_1,t_2)}|\nabla\Psi|^2dx\\&\leq \mathfrak b_0\frac{t_2-t_1+2}{t_2-t_1}\frac{1}{t_2-t_1+2}\int_{\Omega(t_1-1,t_2+1)}|\nabla\Psi|^2dx+\frac{C(t_2-t_1+2)}{(t_1-1)^{2l}(t_2-t_1)}+\frac{C(t_2-t_1+2)^3}{(t_1-1)^{2l+2}(t_2-t_1)}.
\end{split}
\end{equation}
Choose $\bar K$ large such that $\mathfrak b_0\frac{t_2-t_1+2}{t_2-t_1}<\mathfrak b<1$.
Therefore, we have
\begin{equation}\label{t1t2}
\frac{1}{t_2-t_1}\int_{\Omega(t_1,t_2)}|\nabla\Psi|^2dx\leq \mathfrak b\frac{1}{t_2-t_1+2}\int_{\Omega(t_1-1,t_2+1)}|\nabla\Psi|^2dx+\frac{C}{(t_1-1)^{2l}}+\frac{C(t_2-t_1+2)^2}{(t_1-1)^{2l+2}}.
\end{equation}

Let $J$ be an integer satisfying $\frac{T}{2}-1\leq J<\frac{T}{2}$, $t_{1,i}=T-i$ and  $t_{2,i}=T+\bar K+i\ (i=0,1,2\cdots,J)$. It is easy to see that
\begin{equation}
\frac{(t_{2,i}-t_{1,i}+2)^2}{(t_{1,i}-1)^2}\leq C\quad\text{for all $i=0,1,2\cdots,J$}
\end{equation} provided that $T$ is large. Substituting $t_{1,i}$ and $t_{2,i}$ into \eqref{t1t2} yields
\begin{equation}\label{iteration}
\frac{1}{t_{2,i}-t_{1,i}}\int_{\Omega(t_{1,i},t_{2,i})}|\nabla\Psi|^2dx\leq \mathfrak b\frac{1}{t_{2,i+1}-t_{1,i+1}}\int_{\Omega(t_{1,i+1},t_{2,i+1})}|\nabla\Psi|^2dx+\frac{C}{(t_{1,i+1})^{2l}}.
\end{equation}
Iterating the estimate \eqref{iteration} yields
\begin{equation}
\frac{1}{t_{2,0}-t_{1,0}}\int_{\Omega(t_{1,0},t_{2,0})}|\nabla\Psi|^2dx\leq {\mathfrak b}^J\frac{1}{t_{2,J}-t_{1,J}}
\int_{\Omega(t_{1,J},t_{2,J})}|\nabla\Psi|^2dx+\sum\limits_{j=0}^{J}\mathfrak b^j\frac{C}{(t_{1,j})^{2l}}.
\end{equation}
Since $|\nabla\Psi|$ is bounded and $t_{1,j}>t_{1,J}=t_1-J=T-J>\frac{T}{2}$, one has
\begin{equation}
\frac{1}{\bar K}
\int_{\Omega(T,T+\bar K)}|\nabla\Psi|^2dx\leq C\mathfrak b^J+\frac{C}{T^{2l}}.
\end{equation}
When $T$ is large, by the definition of $J$, $\mathfrak b^J\leq\frac{C}{T^{2l}}$ always holds.
Hence
\begin{equation}
\frac{1}{\bar K}
\int_{\Omega(T,T+\bar K)}|\nabla\Psi|^2dx\leq\frac{C}{T^{2l}}.
\end{equation}
Therefore, we have
\begin{equation}
\int_{\Omega(T,T+1)}|\nabla\Psi|^2dx\leq\int_{\Omega(T,T+\bar K)}|\nabla\Psi|^2dx\leq\frac{C \bar K}{T^{2l}}\leq\frac{C}{T^{2l}}.
\end{equation}

\subsection{
 Pointwise  convergence} In this section, we use Nash-Moser iteration to estimate $\|\nabla\Psi\|_{L^\infty(\Omega(T,T+1))}$ in terms of the local energy $L^2$-estimate obtained in Sections \ref{3.1} and \ref{3.2}. This basic idea for oblique derivative boundary value problem for elliptic equation was used in \cite {Lieberman1983The,Du2011}. The key issue is the estimate near the boundary.
 
 For any point $ \bar x=(\bar x_1,\bar x_{2},\bar x_{3})\in\partial{\Omega}$ with $\bar x_3>0$ sufficiently large, let $ x_i= x_i(y_1,y_2)\in C^{2,\alpha}$ $(i=1,2,3)$ be the standard parametrization of $\partial\Omega$ in a small neighborhood of $\bar x$. Suppose that $\bf{n}$ is the unit outer normal vector satisfying
 \begin{equation}
 \cos({\bf{n},x_i})=n_i(y_1,y_2)\in C^{1,\alpha} \quad\text{for $i=1,2,3$}.
 \end{equation}

 Define the map $\mathcal{M}_{y}:y\rightarrow x$ as follows
 \begin{equation}
 x_i= x_i(y_1,y_2)+y_3^{-1}\int_{y_1}^{y_1+y_3}\int_{y_2}^{y_2+y_3}n_i(\alpha_1,\alpha_2)d\alpha_1d\alpha_2,\quad\text{for $i=1,2,3$}.
 \end{equation}
  Then the map, $\mathscr T_{\bar x}=\mathcal{M}_y^{-1}:x\rightarrow y$ to make the boundary flat and satisfies
\begin{equation*}
\mathscr T_{\bar x}(U_{\delta}\cap{\Omega})\rightarrow B^+_{R}\quad\text{and}\quad \mathscr T_{\bar x}(\partial U_{\delta}\cap{\Omega})\rightarrow \partial B^+_{R}\cap\{y_3=0\},
\end{equation*}
where $U_{\delta}$ is a neightborhood of $\bar x$, $B_R^+=\{y_1^2+y_2^2+y_3^2<R,\ y_3>0\}$, $\delta$ and $R$ are uniform constants along the boundary of $\partial\Omega$. Denote the Jacobian $\big(\frac{\partial y_i}{\partial x_j}\big)=D(x)$, then for any $\xi\in\mathbb{R}^3$, there exists a constant $C$ such that
\begin{equation}
C^{-1}|\xi|\leq|D(x)\xi|\leq C|\xi|\quad\text{and}\quad C^{-1}|\xi|\leq|D^{-1}(x)\xi|\leq C|\xi|.
\end{equation}Moreover, on the boundary $\partial U_\delta\cap\partial\Omega$ the map $\mathscr T_{\bar x}$ also satisfies
\begin{equation}
\frac{\partial y_j}{\partial x_i}\frac{\partial y_3}{\partial x_i}=0 \text{ for }j=1,2;\quad\text{and}\quad \big(\frac{\partial y_3}{\partial x_1},\frac{\partial y_3}{\partial x_2},\frac{\partial y_3}{\partial x_3}\big)\times {\textbf{n}}=0.\end{equation}
After changing variables, the problem \eqref{cylinderPsieq} becomes
\begin{equation}\label{Y_equation}
\begin{cases}
\frac{\partial}{\partial y_s}\bigg(\tilde a_{ij}(y)\frac{\partial \Psi}{\partial y_l}\frac{\partial y_l}{\partial x_j}\bigg)\frac{\partial y_s}{\partial x_i}=0&\quad\text{ in} B_R^+,\\[3mm]
\frac{\partial\Psi}{\partial y_s}\frac{\partial y_s}{\partial x_i}\frac{\partial y_3}{\partial x_i}=-\bar q \mathcal W\tilde n_3 &\quad\text{ on }  B_R^+\cap\{y_3=0\},
\end{cases}
\end{equation}
where $\mathcal W=\bigg(\sum\limits_{i=1}^3|\frac{\partial y_3}{\partial x_i}|^2\bigg)^{\frac{1}{2}}$, $\tilde a_{ij}$ and $\tilde n_3$ are the functions $a_{ij}$ and $n_3$ in $y$-coordinates, respectively. In fact, one also has
 \begin{equation}\label{tilde aij}\tilde a_{ij}\frac{\partial \Psi}{\partial y_l}\frac{\partial y_l}{\partial x_j}\frac{\partial y_3}{\partial x_i}=-\rho(\bar q^2)\bar q  \tilde n_3\mathcal W \quad\text{ on }  B_R^+\cap\{y_3=0\}.\end{equation}
For any $\varphi\in C_0^3(B^+_R)$, multiplying $\varphi$ on  both sides of \eqref{Y_equation} and integrating by parts yield
\begin{equation}\label{aij}
\int_{B_{R}^+}\tilde a_{ij}\frac{\partial \Psi}{\partial y_l}\frac{\partial y_l}{\partial x_j}\frac{\partial y_s}{\partial x_i}\frac{\partial\varphi}{\partial y_s}dy+\int_{B_R^+}\tilde a_{ij}\frac{\partial \Psi}{\partial y_l}\frac{\partial y_l}{\partial x_j}\frac{\partial^2y_s}{\partial y_s\partial x_i}\varphi dy=\int_{ B_{R}^+\cap\{y_3=0\}}\tilde a_{ij}\frac{\partial \Psi}{\partial y_l}\frac{\partial y_l}{\partial x_j}\frac{\partial y_3}{\partial x_i}\varphi dy_1dy_2.
\end{equation}
Denote $A_{sl}=\tilde a_{ij}\frac{\partial y_s}{\partial x_j}\frac{\partial y_l}{\partial x_i}$. Obviously, one has
\begin{equation}
\lambda|\xi|^2\leq A_{sl}\xi_s\xi_l\leq\Lambda|\xi|^2\quad \text{and}\quad \bigg|\frac{\partial A_{ls}}{\partial y_r}\bigg|\leq C.
\end{equation}
 It follows from \eqref{tilde aij} and \eqref{aij} that
\begin{equation}\label{Psi}
\int_{B^+_R}A_{ls}\frac{\partial\Psi}{\partial y_l}\frac{\partial\varphi}{\partial y_s}dy=\int_{ B_R^+\cap\{y_3=0\}}A_{l3}\frac{\partial\Psi}{\partial y_l}\varphi dy_1dy_2=-\int_{ B_R^+\cap\{y_3=0\}}\varphi\rho(\bar q^2)\bar q  \tilde n_3\mathcal W dy_1dy_2.
\end{equation}\label{partialA}

Denote $g=\frac{-\rho(q^2)q \tilde  n_3\mathcal W}{A_{33}}$. The definition of $A_{33}$ shows
 \begin{equation}
 \lambda \mathcal W^2\leq \tilde a_{ij}\frac{\partial y_3}{\partial x_i}\frac{\partial y_3}{\partial x_j}=A_{33}\leq \Lambda \mathcal W^2.
\end{equation}
It follows from the assumption \eqref{boundary_algebratic_rate} that 
\begin{equation}
\|g(y')\|_{C^2( B_R^+\cap\{y_3=0\})} \leq \frac{C}{\bar x^{l+1}_3},
\end{equation} where $y'=(y_1,y_2).$
Given $\varsigma(z')\in C^2_0(\mathbb R^2)$ satisfying $\int_{\mathbb R^2}\varsigma(z') dz'=1$, define
\begin{equation}\label{vartheta1}
\vartheta(y)=y_3\int _{\mathbb R^2}g(y'-y_3z')\varsigma(z')dz'.
\end{equation}
Then
\begin{equation}\label{varthetaboundary}
\vartheta (y',0)=\frac{\partial\vartheta}{\partial y_1}(y',0)=\frac{\partial\vartheta}{\partial y_2}(y',0)=0 \quad \text{and}\quad \frac{\partial\vartheta}{\partial y_3}(y',0)=g(y').
\end{equation}
The straightforward computations yield
\begin{equation}\label{vartheta}
\|\vartheta\|_{C^2{(B_R^+)}}\leq \frac{C}{\bar x_3^{l+1}}.
\end{equation}
Define
$\kappa=\partial_s( A_{sl}\partial_l\vartheta)$. It follows from \eqref{vartheta} and the definition of $A_{sl}$ that 
\begin{equation}\label{f}
\|\kappa\|_{L^{\infty}( B_{R}^+)}\leq \frac{C}{\bar x_3^{l+1}}
\end{equation}
and
\begin{equation}\label{vartheta_estimate}-\int_{B_R^+}A_{ls}\partial_l\vartheta\partial_s\varphi+\int_{ B_{R}^+\cap\{y_3=0\}}A_{l3}\partial_l \vartheta\varphi dx=\int_{B_R^+}\kappa\varphi dx\quad \text{for any }\varphi\in C_0^3(B^+_R).\end{equation}
Combining \eqref{Psi} and \eqref{vartheta_estimate} yields
\begin{equation}\label{u}
\int_{B^+_R}A_{ls}\frac{\partial(\Psi-\vartheta)}{\partial y_l}\frac{\partial\varphi}{\partial y_s}dy=\int_{B_R^+}\kappa\varphi dy,
\end{equation} where the boundary term is eliminated because of \eqref{varthetaboundary}.
Denote $v=\Psi-\vartheta$. Replacing $\varphi$ by each $\frac{\partial\varphi}{\partial y_i}\ (i=1,2,3)$ in \eqref{u} and integrating by parts yield
\begin{equation}\label{u_estimate}
\begin{split}
\quad&\int_{ B_R^+}\kappa\frac{\partial \varphi}{\partial y_i}dy=-\int_{B_R^+}A_{ls}\frac{\partial}{\partial y_l}\bigg(\frac{\partial v}{\partial y_i}\bigg)\frac{\partial\varphi}{\partial y_s}dy\\&-\int_{B_R^+}\frac{\partial A_{ls}}{\partial y_i}\frac{\partial v}{\partial y_l}\frac{\partial\varphi}{\partial y_s}dy+\delta_{i3}\int_{ B_{R}^+\cap\{y_3=0\}}A_{ls}\frac{\partial v}{\partial y_l}\frac{\partial\varphi}{\partial y_s}ds,\quad\text{for $i=1,2,3$}.
\end{split}
\end{equation}

Define
\begin{equation*}
\Theta=\max\limits_{ B_R^+\cap\{y_3=0\}}\bigg|\frac{\partial v}{\partial y_3}\bigg|,\quad w_1=\frac{\partial v}{\partial y_1},\quad w_2=\frac{\partial v}{\partial y_2},\quad w_3=\frac{\partial v}{\partial y_3}-\Theta.\end{equation*}
 It follows from \eqref{vartheta} that 
  \begin{equation}\label{Theta}
\Theta\leq\max\limits_{ B_R^+\cap\{y_3=0\}}\bigg|\frac{\partial \Psi}{\partial y_3}\bigg|+\max\limits_{ B_R^+\cap\{y_3=0\}}\bigg|\frac{\partial \vartheta}{\partial y_3}\bigg|\leq \frac{C}{\bar x_3^{l+1}}.\end{equation}
Moreover, the equality \eqref{u_estimate} can be written as
\begin{equation}
\label{uw_estimate}
\begin{split}
&\quad\int_{B_R^+}A_{ls}\frac{\partial w_i}{\partial y_l}\frac{\partial\varphi}{\partial y_s}dy+\int_{B_R^+}w_l\frac{\partial A_{ls}}{\partial y_i}\frac{\partial \varphi}{\partial y_s}dy\\[3mm]
&=\delta_{i3}\int_{ B_{R}^+\cap\{y_3=0\}}A_{ls}\frac{\partial v}{\partial y_l}\frac{\partial\varphi}{\partial y_s}ds-\int_{ B_R^+}\kappa\frac{\partial \varphi}{\partial y_i}dy-\delta_{l3}\Theta\int_{B_R^+}\frac{\partial A_{ls}}{\partial y_i}\frac{\partial \varphi}{\partial y_s}dy,\quad(i=1,2,3).
\end{split}
\end{equation}

Now we use Nash-Moser iteration to get the $L^\infty-$norm of $w_i$. We consider only the case $w_i\geq 0$. If  $w_i\geq0$ does not hold, one can repeat the following proof for $w_i^+$ and $w_i^-$, respectively. It is easy to see that
\begin{equation}\label{w3}
w_3=0 \quad\text{on $B_{R}^+\cap\{y_3=0\}$}.\end{equation}
 For $i=1,2,3$, denote $\varphi_i=\eta^2{w}_i^{\mu+1}$ with some $\mu\geq0$ and some nonnegative function $\eta\in C^2_0({B^+_R})$.
   Direct calculations give
\begin{equation*}
\frac{\partial\varphi_i}{\partial y_r}=2\eta\frac{\partial \eta}{\partial y_r}{w}_i^{\mu+1}+\eta^2(\mu+1)\frac{\partial w_i}{\partial y_r}{w_i}^\mu, \quad\text{for $i=1,2,3$}.
\end{equation*}
Replacing $\varphi$ by $\varphi_i$ for $(i=1,2,3)$ in \eqref{uw_estimate} yields
\begin{equation}
\begin{split}
&\quad\int_{B_R^+}A_{ls}\frac{\partial w_i}{\partial y_l}\frac{\partial\varphi_i}{\partial y_s}dy+\int_{B_R^+}w_l\frac{\partial A_{ls}}{\partial y_i}\frac{\partial \varphi_i}{\partial y_s}dy\\
&=-\int_{ B_R^+}\kappa\frac{\partial \varphi_i}{\partial y_i}dy-\Theta\int_{B_R^+}\frac{\partial A_{3s}}{\partial y_i}\frac{\partial \varphi_i}{\partial y_s}dy\\[2mm]
&=-\int_{B_R^+}(\delta_{is}\kappa+\Theta\frac{\partial A_{3s}}{\partial y_i})\frac{\partial\varphi_i}{\partial y_s}dy, \quad\text{for $i=1,2,3$},\end{split}
\end{equation}
where the boundary term vanishes due to \eqref{w3}.
The straightforward computations give
\begin{equation}\label{A_estimate}
\begin{split}
&\quad\int_{B_R^+}A_{ls}\frac{\partial w_i}{\partial y_l}\bigg(2\eta\frac{\partial \eta}{\partial y_s}{w}_i^{\mu+1}+\eta^2(\mu+1)\frac{\partial w_i}{\partial y_s}{w_i}^\mu\bigg) dy\\[2mm]
&\geq\lambda(\mu+1)\int_{B_R^+}\eta^2w_i^\mu|Dw_i|^2
dy-2\Lambda\int_{B_R^+}\eta w_i^{\mu+1}|D\eta||Dw_i|dy\\[2mm]&\geq\lambda(\mu+1)\int_{B_R^+}\eta^2w_i^\mu|Dw_i|^2
dy-\epsilon\int_{B_R^+}\eta^2w_i^\mu|Dw_i|^2-\frac{1}{\epsilon}\int_{B_R^+}|D\eta|^2w_i^{\mu+2}dy
\end{split}
\end{equation}
and
\begin{equation}\label{w_A_estimate}
\begin{split}
&\quad\int_{B_R^+}w_l\frac{\partial A_{ls}}{\partial y_i}\frac{\partial \varphi_i}{\partial y_s}dy\\&\leq C\int_{B_R^+} \eta|D\eta|w_i^{\mu+1}\bigg|\frac{\partial A_{ls}}{\partial y_i}w_l\bigg|+\bigg|\frac{\partial A_{ls}}{\partial y_i}w_l\bigg|\eta^2(\mu+1)w_i^{\mu}|Dw_i|dy\\
&\leq C\int_{B_R^+}|D\eta|^2w_i^{\mu+2}+\eta^2\bigg|\frac{\partial A_{ls}}{\partial y_i}w_l\bigg|^2w_i^\mu+\epsilon\eta^2(\mu+1)^2w_i^\mu|Dw_i|^2+\frac{1}{\epsilon}\bigg|\frac{\partial A_{ls}}{\partial y_i}w_l\bigg|^2\eta^2w_i^\mu dy\\
&\leq C\int_{B_R^+}|D\eta|^2w_i^{\mu+2}+\eta^2|\bar w|^2w_i^\mu+\epsilon\eta^2(\mu+1)^2w_i^\mu|Dw_i|^2+\frac{1}{\epsilon}|\bar w|^2\eta^2w_i^\mu dy,
\end{split}
\end{equation}
where  $|\bar w|^2=w_1^2+w_2^2+w_3^2$.
Denote \begin{equation*}\mathcal F_{is}=\delta_{is}\kappa+\Theta\frac{\partial A_{3s}}{\partial y_i}\quad\text{and}\quad\mathcal K=\max|\mathcal F_{is}|.\end{equation*} It follows from \eqref{f} and \eqref{Theta} that  \begin{equation}\label{F}
|\mathcal K|\leq \frac{C}{\bar x_3^{l+1}}.
\end{equation}
If $|w_i|\leq \mathcal K$, then we prove the pointwise bounds claimed in Theorem \ref{convergence_thm} . Hence we assume
\begin{equation}\label{wi}|w_i|\geq \mathcal K,\quad (i=1,2,3).\end{equation} Therefore, one has
\begin{equation}\label{F_estimate}\begin{split}
&\quad\int_{B_{R}^+}\mathcal F_{is}\frac{\partial\varphi_i}{\partial y_s}dy\\&\leq \int_{B_R^+}\mathcal K\eta|D\eta|w_i^{\mu+1}+\mathcal K\eta^2(\mu+1)w_i^\mu|Dw_i|dy\\
&\leq  \int_{B_R^+}\eta|D\eta|w_i^{\mu+2}+\epsilon\eta^2(\mu+1)^2w_i^\mu|Dw_i|^2+\frac{1}{\epsilon}\eta^2\mathcal K^2w_i^\mu dy\\
&\leq \int_{B_R^+}\eta|D\eta|w_i^{\mu+2}+\epsilon\eta^2(\mu+1)^2w_i^\mu|Dw_i|^2+\frac{1}{\epsilon}\eta^2w_i^{\mu+2} dy.  \end{split}
\end{equation}
Combining \eqref{A_estimate}, \eqref{w_A_estimate}, and \eqref{F_estimate} yields
\begin{equation}
\begin{split}
&\quad\lambda(\mu+1)\int_{B_R^+}\eta^2w_i^\mu|Dw_i|^2dy- \big(\epsilon+2\epsilon(\mu+1)^2\big)
\int_{ B_R^+}\eta^2w_i^\mu|Dw_i|^2dy\\
&\leq (C+\frac{C}{\epsilon})\int_{ B_R^+}\eta^2|\bar w|^2w_i^\mu dy+(C+\frac{C}{\epsilon})\int_{ B_R^+}(\eta^2+|D\eta|^2)w_i^{\mu+2} dy.
\end{split}\end{equation}
If we choose $\epsilon=\frac{\lambda}{8(\mu+1)}$, then one has
\begin{equation}
\int_{ B_R^+}\eta^2w_i^2|Dw_i|^2dy\leq C\int_{ B_R^+}|D\eta|^2|w_i|^{\mu+2}+\eta^2|w_i|^{\mu+2}+\eta^2|\bar w|^2|w_i|^\mu dy.
\end{equation}
Therefore,
\begin{equation}
\int_{ B_R^+}\bigg|D(\eta w_i^\frac{\mu+2}{2})\bigg|^2dy\leq C(\mu+2)^2\int_{ B_R^+}|D\eta|^2|w_i|^{\mu+2}+\eta^2|w_i|^{\mu+2}+\eta^2|\bar w|^2|w_i|^\mu dy.
\end{equation}
Applying Sobolev inequality yields
\begin{equation}
\bigg(\int_{ B_R^+}(\eta w_i^\frac{\mu+2}{2})^6dy\bigg)^\frac{1}{3}\leq C(\mu+2)^2\int_{ B_R^+}|D\eta|^2|w_i|^{\mu+2}+\eta^2|w_i|^{\mu+2}+\eta^2|\bar w|^2|w_i|^\mu dy.
\end{equation}

Set \begin{equation*}
R_j=(\frac{1}{2}+\frac{1}{2^{j+1}})R\quad\text{and}\quad \gamma_j=2\cdot 3^j.
\end{equation*}
Let $\eta_j\in C_0^\infty(B_{R_j}^+)$ satisfy \begin{equation*}
	\quad \eta_j=1\text{ in } B_{R_{j+1}}^+ \quad\text{and}\quad |D\eta_j|\leq\frac{4}{R_j-R_{j+1}}.
\end{equation*}
Choosing $\mu=\gamma_j-2$ yields
\begin{equation*}
\bigg(\int_{ B_{R_{j+1}}^+}w_i^{\gamma_{j+1}}dy\bigg)^\frac{1}{3}\leq C\gamma_j^2\int_{ B_{R_j}^+}\big(2^{j+1}R\big)^2w_i^{\gamma_j}+w_i^{\gamma_j}+|\bar w|^2w_i^{\gamma_j-2} dy.
\end{equation*}
Thus, one has
\begin{equation}
\bigg(\int_{ B_{R_{j+1}}^+}w_i^{\gamma_{j+1}}dy\bigg)^{\frac{1}{\gamma_{j+1}}}\leq\bigg(\int_{ B_{R_j}^+}A_jw_i^{\gamma_j}+B_jw_i^{\gamma_j}+B_j|\bar w|^2w_i^{\gamma_j-2} dy\bigg)^\frac{1}{\gamma_j},
\end{equation}
where $A_j=C\gamma_j^2\big(2^{j+1}/R\big)^2$ and $B_j=C\gamma_j^2$.
Note that
\begin{equation}
\int_{B_{R_j}^+}|\bar w|^2w_i^{\gamma_j-2}dy\leq\bigg(\int_{ B_{R_j}^+}w_i^{\gamma_j}\bigg)^\frac{\gamma_j-2}{\gamma_j}\bigg(\int_{ B_{R_j}^+}|\bar w|^{\gamma_j}\bigg)^{\frac{2}{\gamma_j}}.
\end{equation}
Therefore, one has
\begin{equation}
\begin{split}
&\quad\bigg(\int_{ B_{R_{j+1}}^+}w_i^{\gamma_{j+1}}dy\bigg)^{\frac{1}{\gamma_{j+1}}}\\
&\leq \bigg[A_j\int_{ B_{R_j}^+}w_i^{\gamma_j}+B_j\int_{ B_{R_j}^+}w_i^{\gamma_j}+B_j\bigg(\int_{ B_{R_j}^+}w_i^{\gamma_j}\bigg)^\frac{\gamma_j-2}{\gamma_j}\bigg(\int_{ B_{R_j}^+}|\bar w|^{\gamma_j}\bigg)^{\frac{2}{\gamma_j}}\bigg]^\frac{1}{\gamma_j}\\[2mm]
&\leq \|w_i\|_{L^{\gamma_j}(B_{R_{j}}^+)}^{\frac{\gamma_j-2}{\gamma_j}}\bigg[(A_j+B_j)\|w_i\|_{L^{\gamma_j}(B_{R_{j}}^+)}^2+B_j\|\bar w\|_{L^{\gamma_j}(B_{R_{j}}^+)}^2\bigg]^{\frac{1}{\gamma_j}}\\[2mm]&\leq\|w_i\|_{L^{\gamma_j}(B_{R_{j}}^+)}^{\frac{\gamma_j-2}{\gamma_j}}(A_j+2B_j)^{\frac{1}{\gamma_j}}\|\bar w\|_{L^{\gamma_j}(B_{R_{j}}^+)}^\frac{2}{\gamma_j}.
\end{split}
\end{equation}
This implies that
\begin{equation}\label{w}
\begin{split}
&\quad\sum\limits_{i=1}^{3}\|w_i\|_{L^{\gamma_{j+1}}(B_{R_{j+1}}^+)}\\&\leq\sum\limits_{i=1}^3\|w_i\|_{L^{\gamma_j}(B_{R_{j}}^+)}^{\frac{\gamma_j-2}{\gamma_j}}(A_j+2B_j)^{\frac{1}{\gamma_j}}\|\bar w\|_{L^{\gamma_j}(B_{R_{j}}^+)}^\frac{2}{\gamma_j}\\
&\leq \bigg[\sum\limits_{i=1}^3\bigg(\|w_i\|_{L^{\gamma_j}(B_{R_{j}}^+)}^{\frac{\gamma_j-2}{\gamma_j}}(A_j+2B_j)^{\frac{1}{\gamma_j}}\|\bar w\|_{L^{\gamma_j}(B_{R_{j}}^+)}^\frac{2}{\gamma_j}\bigg)^{\frac{\gamma_j}{\gamma_j-2}}\bigg]^\frac{\gamma_j-2}{\gamma_j}3^{\frac{2}{\gamma_j}}\\
&\leq(9A_j+18B_j)^\frac{1}{\gamma_j}\bigg(\sum\limits_{i=1}^{3}\|w_i\|_{L^{\gamma_j}(B_{R_{j}}^+)}\bigg)^\frac{2}{\gamma_j}\bigg(\sum\limits_{i=1}^{3}\|w_i\|_{L^{\gamma_j}(B_{R_{j}}^+)}\bigg)^{\frac{\gamma_j-2}{\gamma_j}}\\
&\leq (9A_j+18B_j)^\frac{1}{\gamma_j}\sum\limits_{i=1}^{3}\|w_i\|_{L^{\gamma_j}(B_{R_{j}}^+)}.
\end{split}
\end{equation}

Set \begin{equation*}
\mathcal{Q}_{j+1}=\sum\limits_{i=1}^{3}\|w_i\|_{L^{\gamma_{j+1}}(B_{R_{j+1}}^+)}\quad \text{and} \quad S_j=(9A_j+18B_j)^\frac{1}{\gamma_j}.
\end{equation*}
Hence the estimate \eqref{w} can be written as
\begin{equation}
\mathcal{Q}_{j+1}\leq S_j \mathcal{Q}_j.
\end{equation}
Obviously,
\begin{equation}
S_j=(9A_j+18B_j)^\frac{1}{\gamma_j}\leq \bigg(C\gamma_j^2\big(2^{j+1}/R\big)^2\bigg)^\frac{1}{\gamma_j}\leq C^\frac{1}{\gamma_j}18^{\frac{j}{\gamma_j}}.
\end{equation}
Therefore, one has
\begin{equation}\label{iterationQ}
\mathcal{Q}_{j+1}\leq C^{\sum\limits_{i=1}^j\frac{1}{\gamma_i}}18^{\sum\limits_{i=1}^j\frac{i}{\gamma_i}}\mathcal{Q}_0.
\end{equation}
Note that
\begin{equation*}
\sum\limits_{i=1}^j\frac{1}{\gamma_i}\leq C\quad \text{and}\quad\sum\limits_{i=1}^j\frac{i}{\gamma_i}\leq C.
\end{equation*}
Taking $j\rightarrow\infty$ in \eqref{iterationQ} yields
\begin{equation}
\sum\limits_{i=1}^3\|w_i\|_{L^\infty(B_{\frac{1}{2}R}^+)}\leq C\sum\limits_{i=1}^3\|w_i\|_{L^2(B_R^+)},
\end{equation}
provided that \eqref{wi} holds.
Hence, we have
\begin{equation}\label{nash-moser}
\sum\limits_{i=1}^3\|w_i\|_{L^\infty(B_{\frac{1}{2}R}^+)}\leq C\sum\limits_{i=1}^3\|w_i\|_{L^2(B_R^+)}+\mathcal K.
\end{equation}
It follows from the definition of $w_i$, \eqref{vartheta}, \eqref{f} and \eqref{Theta}, one has
\begin{equation}\label{bdygradientPsi}
	\|\nabla\Psi\|_{L^{\infty}(U_\delta)}\leq C\big(\|\nabla\Psi\|_{L^2(U_\delta)}+\mathcal K+\Theta+\|\vartheta\|_{C^2{(B_R^+)}}\big).
	\end{equation}
	For the interior estimate, as same as the estimate for \eqref{bdygradientPsi} with $\vartheta=0$, $\Theta=0$ and $ \mathcal K=0$ to obtain for any $B_R\in\Omega$, one has
	\begin{equation}\label{interior}
	\|\nabla\Psi\|_{L^\infty(B_{\frac{R}{2}})}\leq C\|\nabla\Psi\|_{L^2(B_R)}.
	\end{equation}
	In a word, when the boundary satisfies \eqref{boundary_algebratic_rate} and $x_3$ is sufficiently large, combining \eqref{bdygradientPsi} and \eqref{interior} yields
	\begin{equation}
	|\nabla\phi_1 (x_1, x_2, x_3)-(0,0,\bar q)|\leq C  x_3^{-{l}}.
	\end{equation}
	
	For the case of the nozzle is a perfect cylinder for $x_3$ sufficient large, obviously the third component of the normal direction at the boundary is zero, i.e. $n_3=0$. Applying the estimate \eqref{bdygradientPsi} with $\vartheta=0$, $\Theta=0$ and $ \mathcal K=0$ yields that there is a positive constant $\mathfrak d$ such that
	\begin{equation}
	\|\nabla\Phi\|_{L^{\infty}(\Omega(T,T+1))}\leq \frac{C}{e^{\mathfrak d T}}.
	\end{equation}
	Hence, the proof of Theorem \ref{convergence_thm} is finished.

\medskip
	\textbf{Acknowledgement.}
	The authors were supported partially by NSFC grants 11971307 and 11631008. Xie is also supported by Young Changjiang Scholar of Ministry of Education in China.
	\nocite{*}
	\bibliographystyle{abbrv}
	\bibliography{ref}

 \end{document}